\documentclass[leqno,10pt]{amsart}
\usepackage{latexsym}
\usepackage{amssymb}
\usepackage{amsfonts}
\usepackage{amsmath}
 \usepackage{todonotes}
\usepackage{color}
\usepackage{enumerate}
\usepackage[normalem]{ulem}

\allowdisplaybreaks
 
\theoremstyle{plain}
\newtheorem{thm}{Theorem}[section]
\newtheorem{lemma}[thm]{Lemma}
\newtheorem{prop}[thm]{Proposition}
\newtheorem{corollary}[thm]{Corollary}
\theoremstyle{definition}

\newtheorem{remark}{Remark}

\newtheorem{ex}[thm]{Example}

\def\d{{\fam0 d}}

\def\loc{\operatorname{loc}}
 \def\supp{\operatorname{supp}}
 \def\BMO{\operatorname{BMO}}

\def\Xint#1{\mathchoice
   {\XXint\displaystyle\textstyle{#1}}%
   {\XXint\textstyle\scriptstyle{#1}}%
   {\XXint\scriptstyle\scriptscriptstyle{#1}}%
   {\XXint\scriptscriptstyle\scriptscriptstyle{#1}}%
   \!\int}
\def\XXint#1#2#3{{\setbox0=\hbox{$#1{#2#3}{\int}$}
     \vcenter{\hbox{$#2#3$}}\kern-.5\wd0}}
\def\dashint{\Xint-}

\newtoks\by
\newtoks\paper
\newtoks\book
\newtoks\jour
\newtoks\yr
\newtoks\pages
\newtoks\vol

\newtoks\publ
\newtoks\eds
\newtoks\proc
\def\ota{{\hbox{???}}}

\def\cLear{\by=\ota\paper=\ota\book=\ota\jour=\ota\yr=\ota
\pages=\ota\vol=\ota\publ=\ota}

\def\endpaper{\the\by, \textit{\the\paper},
{\the\jour}  {\the\vol} (\the\yr), \the\pages.\cLear}

\def\endbook{\the\by, \textit{\the\book}, \the\publ.\cLear}
\def\endprep{\the\by, \textit{\the\paper}, \the\jour.\cLear}
\def\endproc{\the\by, \textit{\the\paper}, \the\publ, \the\pages.\cLear}
\def\name#1#2{#1 #2}

\numberwithin{equation}{section}

\begin{document}

\title{Higher-order Sobolev embeddings into spaces of Campanato and  Morrey type}

\begin{abstract} Necessary and sufficient conditions are offered for  Sobolev type spaces built on 
  re\-ar\-ran\-ge\-ment-in\-va\-ri\-ant spaces   
to be continuously
embedded into (generalized) Campanato and Morrey spaces on 
  open subsets of the $n$-dimensional Euclidean space. As a consequence, the optimal target and domain spaces in the relevant embeddings are identified. Our general criteria are implemented to derive sharp embeddings in the class of Orlicz-Sobolev spaces.
\end{abstract}

\author{Paola Cavaliere, Andrea Cianchi,  Lubo\v s Pick and Lenka Slav\'{\i}kov\'a}

\address{Dipartimento di Matematica\\ Universit\`a di Salerno\\ Via Giovanni Paolo II, 84084 Fisciano (SA), Italy}
 \email{pcavaliere@unisa.it}

\address{Dipartimento di Matematica e Informatica \lq \lq U. Dini''\\
Universit\`a di Firenze\\  Viale Morgagni 57/A, 50134 Firenze,
Italy}
 \email{cianchi@unifi.it}

\address{Department of Mathematical Analysis
\\  Faculty of Mathematics and Physics,
Charles University
\\  Sokolovsk\'a~83,
186~75 Praha~8,
Czech Republic}
\email{pick@karlin.mff.cuni.cz}

\address{Department of Mathematical Analysis
\\  Faculty of Mathematics and Physics,
Charles University
\\  Sokolovsk\'a~83,
186~75 Praha~8,
Czech Republic}
\email{slavikova@karlin.mff.cuni.cz}



\maketitle

\section{Introduction}\label{Intro}
Morrey and Campanato spaces of functions defined in open sets of $\mathbb R^n$, with $n \geq 2$, were introduced in view of applications to the regularity theory of solutions to partial differential equations. They are nowadays
 ubiquitous in the modern developments of this theory.   Loosely speaking, Morrey norms measure the degree of decay of norms of functions over balls when their radius approaches zero, whereas Campanato (semi-)norms provide information on the degree of decay over balls of the oscillation of functions in norm. Unlike several customary Banach spaces of functions, these two families of spaces are not rearrangement-invariant, since their norms or semi-norms do not just depend on global integrability properties of functions. The vast literature on various aspects of the theory of these spaces includes the papers \cite{AX, AlS, AlS2, BDVS,  Bur:24, Bur:22, BuDa, Ca0, Ca, DaGi,  DDY, Dev:84, HO, JN, Man:21, Samko} and the  monographs \cite{Korenobook, Stbook, Trie}.

The present paper addresses the question of how integrability properties of weakly differentiable functions are reflected in their membership in a Morrey or a Campanato space. More precisely, we are concerned with embeddings of Sobolev spaces built upon general rearrangement-invariant norms into Morrey and Campanato type spaces.  The results to be presented complement earlier contributions in the same vein, where a comprehensive picture of embeddings of rearrangement-invariant norm-based Sobolev spaces into rearrangement-invariant spaces~\cite{CPS, EKP, KP1} and spaces of uniformly continuous functions~\cite{Monia1} was offered.

Embeddings of classical Sobolev spaces, corresponding to the case when the defining re\-ar\-ran\-ge\-ment-invariant norms are of Lebesgue type, into Morrey and Campanato spaces are available in the literature. On the other hand, little seems to be known when less standard rearrangement-invariant norms, such as Orlicz norms, are considered.

Dealing with embeddings for Sobolev spaces of such a general form naturally calls for Morrey and Campanato target spaces from broader classes than those customarily employed. These extended classes include Morrey and Campanato type spaces whose norms are defined via a decay on balls which is not necessarily polynomial in their radius. Generalized Morrey and Campanato spaces were introduced and thoroughly analyzed in~\cite{Sp}. They are of use, for instance, in the description of sharp assumptions and regularity properties of solutions to nonlinear elliptic problems -- see e.g. \cite{BCDS, CW, CiSchw}. The function that defines the relevant general decay will be denoted by~$\varphi$, and the corresponding Morrey and Campanato spaces on an open set $\Omega \subset \mathbb R^n$ by $\mathcal M^{\varphi(\cdot)}(\Omega)$ and $\mathcal L^{\varphi(\cdot)}(\Omega)$, respectively. In the case when $\varphi$ is a power, these spaces reproduce the classical ones.  The notation~$W^mX(\Omega)$ is employed for the $m$-th order Sobolev space of those functions which belong, together with their derivatives up to the order~$m$, to the rearrangement-invariant space~$X(\Omega)$.

A basic version of the embeddings under examination reads
\begin{equation}\label{nov1}
W^mX(\Omega) \to \mathcal M^{\varphi(\cdot)} (\Omega)
\end{equation}
and
\begin{equation}\label{nov2}
W^mX(\Omega) \to \mathcal L^{\varphi(\cdot)}(\Omega).
\end{equation}
Necessary and sufficient conditions on $m$, $X(\Omega)$ and $\varphi$ for~\eqref{nov1} and~\eqref{nov2} to hold are exhibited under minimal regularity assumptions on $\Omega$. As a consequence, fixing $m$, they enable us to detect the optimal target~$\mathcal M^{\varphi(\cdot)} (\Omega)$ in~\eqref{nov1} and~$\mathcal L^{\varphi(\cdot)}(\Omega)$ in~\eqref{nov2} for a given domain space~$X(\Omega)$, and conversely, the optimal domain~$X(\Omega)$ for given targets~$\mathcal M^{\varphi(\cdot)} (\Omega)$ and~$\mathcal L^{\varphi(\cdot)}(\Omega)$.

These questions in the case of first-order Sobolev spaces were discussed in~\cite{CPcamp}. The focus here is on the higher-order case when $m \geq 2$. The conclusions in this range of values of $m$ disclose new traits, which cannot be seen in the first-order case. 

Novelties also surface when embeddings into Campanato spaces~$\mathcal L^{k, \varphi(\cdot)}(\Omega)$ of higher integer order~$k$  are taken into account. The latter are defined in terms of the decay on balls of the deviation in norm from polynomials of degree at most~$k$, the basic spaces~$\mathcal L^{\varphi(\cdot)}(\Omega)$ being thus recovered for $k=0$.
The same problems outlined with regard to the embedding~\eqref{nov2} are tackled and solved  in connection with embeddings of the form
\begin{equation}\label{nov3}
W^mX(\Omega) \to \mathcal L^{k, \varphi(\cdot)}(\Omega).
\end{equation}

Some features of our results are worth being pointed out.
 The embeddings into the Morrey spaces~$\mathcal M^{\varphi(\cdot)}(\Omega)$ are equivalent to embeddings into Marcinkiewicz spaces associated with the same function $\varphi$. The latter spaces are defined analogously to~$\mathcal M^{\varphi(\cdot)}(\Omega)$, save that the role of balls is played by arbitrary measurable sets. Such an equivalence is non-obvious, since Marcinkiewicz spaces are strictly contained in the Morrey spaces modeled upon the same function~$\varphi$, except for very specific choices of $\varphi$.
 
  In contrast to the case of classical Sobolev and Campanato spaces, an embedding of the form~\eqref{nov2}, with $\varphi$ vanishing at zero, does not necessarily imply a corresponding embedding into the space of continuous functions with modulus of continuity~$\varphi$. Moreover, even if~$W^mX(\Omega)$ is embedded into a space of continuous functions, the optimal modulus of continuity can be different (weaker) than~$\varphi$. These peculiarities typically emerge when the function $\varphi$ has a very slow decay at zero -- of logarithmic type, for instance.

Sobolev embeddings into vanishing versions of the Morrey and Campanato spaces are described as well. These spaces are defined by replacing the boundedness with the decay to zero, as the radius of balls approaches zero, in the definitions of Morrey and Campanato spaces. In the case of Campanato spaces, they generalize the notion of the  Sarason space ${\rm VMO}(\Omega)$ of functions of vanishing mean oscillation, corresponding to the choice $\varphi=1$, to an arbitrary function $\varphi$.

Finally, the principles established in a general framework are implemented to derive optimal embeddings in the important instance of Orlicz-Sobolev spaces, which arise when $X(\Omega)$ is an Orlicz space. Although not the most general one, this class of Sobolev spaces is much richer than that of the standard Sobolev spaces and apt to reveal the novel peculiarities we alluded to above.

\section{Main results}\label{mainres}

Our embeddings hold in any bounded John domain in $\mathbb R^n$, with $n \geq 2$. They constitute a broad class where customary Sobolev type embeddings hold.
Recall that an open set $\Omega \subset \mathbb R^n$ is called 
a~\emph{John domain}~if there exist positive constants $c$ and $\ell$, and a
point $x_0 \in \Omega$ such that for every $x \in \Omega$ there exists a
rectifiable curve $\gamma \colon [0, \ell] \to \Omega$, parameterized by
arclength, such that $\gamma (0)=x$, $\gamma (\ell) = x_0$, and
$${\rm dist}\, (\gamma (r) , \partial \Omega ) \geq c \, r \quad \ 
\hbox{for} \ r \in [0, \ell].$$
Any bounded open set
satisfying the uniform
interior cone condition is a bounded John domain. In particular, bounded Lipschitz domains are bounded John domains.
For simplicity of notation, throughout the paper we assume, unless otherwise stated, that $|\Omega|=1$, where $|\cdot|$ stands for the Lebesgue measure in $\mathbb R^n$.

We denote by $X(\Omega)$ the rearrangement-invariant space on $\Omega$ associated with a
rearrangement-invariant function norm $\|\cdot\|_{X(0,1)}$. Roughly speaking, the latter is a functional, defined on the set $L^0_+(0,1)$ of nonnegative measurable functions $f$ in $(0,1)$, which enjoys standard properties of norms and only depends on their decreasing rearrangement~$f^*$.
As already mentioned in Section~\ref{Intro}, given $m \in \mathbb N$, the notation $W^mX(\Omega)$ stands for the Sobolev space of those functions which belong, together with their weak derivatives up to the order $m$, to~$X(\Omega)$.  More precise definitions can be found in the next section.

The generalized Morrey and Campanato spaces  are built upon  \emph{admissible}~functions
 $\varphi\colon(0,\infty)\to(0,\infty)$, in the sense that
\begin{equation}\label{E:admissible-new}
    \inf_{r \in [a,\infty)} \varphi(r)>0\quad \ 
\hbox{for} \ a\in (0,\infty).
\end{equation}
Their definition is given below. Let us point out here that
in the  classical framework when~$\varphi$ is a power,   either~$\mathcal M^{\varphi(\cdot)}(\Omega) =\mathcal L^{\varphi(\cdot)}(\Omega)$ (negative power) or~$\mathcal M^{\varphi(\cdot)}(\Omega)=\{0\}$ (positive power). 
Moreover, in the latter case, $\mathcal L^{\varphi(\cdot)}(\Omega)$  agrees with a space of H\"older continuous functions.
The only nontrivial relation between the two spaces associated with the same power occurs in the borderline case when such a  power is $0$, and hence $\varphi =1$, in which case $\mathcal L^{\varphi(\cdot)}(\Omega)= {\rm BMO}(\Omega) \supsetneq L^\infty(\Omega) = \mathcal M^{\varphi(\cdot)}(\Omega)$, where ${\rm BMO}(\Omega)$ denotes the John-Nirenberg space of functions with bounded mean oscillation.

The use of more general admissible functions $\varphi$ uncovers a refined picture of possible inclusion relations of Morrey, Campanato and H\"older type spaces. This is apparent when admissible functions $\varphi$ with a slow decay or growth near $0$ are involved. In particular, as already mentioned in Section~\ref{Intro}, the mere vanishing of $\varphi$ at $0$ does not guarantee that~$\mathcal L^{\varphi(\cdot)}(\Omega)$ only consists of continuous, or even bounded, functions. This issue is discussed in Subsection~\ref{S:comp}. In this connection, see also 
 the papers~\cite{Sp} and~\cite{CPcamp}.

\subsection{Embeddings into Morrey spaces}
\label{S:morrey}

Given an admissible function $\varphi$, we define the \emph{Morrey space} $\mathcal M^{\varphi(\cdot)}(\Omega)$ as the collection of all measurable functions $u: \Omega \to \mathbb R$  which make the norm
\begin{equation}\label{morreydef}
    \|u\|_{\mathcal M^{\varphi(\cdot)}(\Omega)}=\sup_{B\subset \Omega} \, \frac{1}{\varphi(|B|^{\frac{1}{n}})} \, \dashint_{B}|u|\,\d x
\end{equation}
finite. Here, $B$ denotes an open ball in $\mathbb R^n$,   and $\dashint$ stands for integral average.   When \begin{align}\label{power}
    \varphi (r)=r^\alpha,
\end{align} we shall use the abridged notation $\mathcal M^\alpha (\Omega)$
for $\mathcal M^{\varphi(\cdot)}(\Omega)$. Let us warn that  the presence of the integral average on the right-hand side of \eqref{morreydef} causes a shift by  $n$ in the exponent of the space $\mathcal M^\alpha (\Omega)$ if compared with the customary notation for classical Morrey spaces.

 As already mentioned in Section~\ref{Intro}, Morrey spaces share the same properties as Marcinkiewicz spaces, associated with the same function $\varphi$, as targets for embeddings of $W^mX(\Omega)$. Recall that the
Marcinkiewicz function norm $\|\cdot\|_{\mathfrak M^{\varphi(\cdot)}(\Omega)}$ associated with $\varphi$ is obtained by just replacing balls by arbitrary measurable sets $E\subset \Omega$ on the right-hand side of equation~\eqref{morreydef}. Namely,
\begin{equation}\label{marcdef}
    \|u\|_{\mathfrak M^{\varphi(\cdot)}(\Omega)}=\sup_{E\subset \Omega}  \, \frac{1}{\varphi(|E|^{\frac{1}{n}})}\, \dashint_{E}|u|\,\d x.
\end{equation}
Unlike $\mathcal M^{\varphi(\cdot)}(\Omega)$, the space $\mathfrak M^{\varphi(\cdot)}(\Omega)$ is always rearrangement-invariant. Also,
\begin{equation}\label{marc-mor}
\mathfrak M^{\varphi(\cdot)}(\Omega)\to\mathcal M^{\varphi(\cdot)}(\Omega),\end{equation}
and, for most choices of $\varphi$, one has that
$\mathfrak M^{\varphi(\cdot)}(\Omega) \neq \mathcal M^{\varphi(\cdot)}(\Omega)$.  
When  $\varphi$ is of power type \eqref{power} for some  $\alpha \in (-n,0)$, then  $\mathfrak M^{\varphi(\cdot)}(\Omega)=L^{-\frac n    \alpha, \infty}(\Omega)$,  the weak Lebesgue space.  If $\alpha \leq -n$, then $\mathfrak M^{\varphi(\cdot)}(\Omega)=L^1(\Omega)$, and if $\alpha =0$, then $\mathfrak M^{\varphi(\cdot)}(\Omega)=L^\infty(\Omega)$. When $\alpha >0$, trivially $\mathfrak M^{\varphi(\cdot)}(\Omega)=\{0\}$.

\medskip

We begin with a necessary and sufficient condition for an arbitrary-order Sobolev space to be embedded into a Morrey space.

\begin{thm}{\rm\bf{ [Criterion for embeddings into Morrey spaces]}}\label{T:2}
    Assume that $\Omega$ is a bounded John domain in $\mathbb R^n$, and  $m \in \mathbb N$ is such that $m \leq n -1$.   Let $\|\cdot \|_{X(0,1)}$ be a~rearrangement-invariant function norm, and let $\varphi$ be an admissible function.
    Then, the following statements are equivalent.
    \begin{enumerate}[(i)]
        \item The  embedding
         \begin{equation}\label{E:SO-MAemb}
            W^{m}X(\Omega) \to \mathcal M^{\varphi(\cdot)}(\Omega)
        \end{equation}
        holds.
        \item The embedding
\begin{equation}\label{E:SO-MARCemb}
            W^{m}X(\Omega) \to  
            \mathfrak M^{\varphi(\cdot)}(\Omega)
        \end{equation}
        holds.
        \item   We have that
         \begin{equation}\label{E:7}
            \sup_{r \in (0,1)} \, \frac{1}{\varphi( r)}\left\| s^{-1+\frac{m}{n}}\chi_{{(r^n,1)}}(s)\right\|_{{X'(0,1)}} <\infty.
        \end{equation}
    \end{enumerate}
\end{thm}

\begin{remark}\label{rem-trivialmorrey}{\rm
Although Theorem~\ref{T:2} holds without any restriction on $m$, it becomes trivial when $m \geq n$. Indeed,
since $X(\Omega) \to L^1(\Omega)$ for any rearrangement-invariant space $X(\Omega)$, the standard Sobolev embedding theorem implies that $W^mX(\Omega) \to L^{\infty}(\Omega)$~  \cite[Theorem~4.12]{AF}. On the other hand, $\big\| s^{-1+\frac{m}{n}}\big\|_{{X'(0,1)}} <\infty$ under the current assumption on $m$. Therefore,
Theorem~\ref{T:2} amounts to the equivalence of embeddings~\eqref{E:SO-MAemb} and~\eqref{E:SO-MARCemb} to the fact that $\inf_{r>0}{\varphi(r)}>0$. The latter condition is necessary and sufficient  for $\mathcal M^{\varphi(\cdot)}(\Omega)\neq\{0\}$  and $\mathfrak M^{\varphi(\cdot)}(\Omega) \neq \{0\}$. Thus, Theorem~\ref{T:2} tells us that  embeddings~\eqref{E:SO-MAemb} and~\eqref{E:SO-MARCemb} hold for every rearrangement-invariant space $X(\Omega)$ and every nontrivial Morrey and Marcinkiewicz space
$\mathcal M^{\varphi(\cdot)}(\Omega)$  and $\mathfrak M^{\varphi(\cdot)}(\Omega)$, respectively.}
\end{remark}

\begin{remark}\label{rem-nov7}{\rm If $\inf_{r>0}{\varphi(r)}=0$, then condition~\eqref{E:7} fails for any   $\Omega$, $n$,   $m$,   and $\|\cdot\|_{X(0,1)}$. This is consistent with the fact that, in this case, $\mathcal M^{\varphi(\cdot)}(\Omega)=\{0\}$.}
\end{remark}

\begin{remark}\label{rem-nov6}{\rm
Replacing $\sup_{r \in (0,1)}$ with  $\sup_{r \in (0,r_0)}$ in~\eqref{E:7}, for any $r_0\in (0,1)$, results in an equivalent condition.}
\end{remark}

\begin{remark}\label{rem-m=1}{\rm
The equivalence of assertions {\it(i)} and {\it(iii)} in Theorem \ref{T:2} was established in \cite[Theorem 1.2]{CPcamp} in the case when $\Omega$
is a cube  and $m=1$.
}
\end{remark}

The next two results rely upon Theorem~\ref{T:2} and provide us with the optimal target and the optimal domain in embeddings of the form~\eqref{E:SO-MAemb}.

\begin{thm}{\rm\bf{ [Optimal Morrey target]}} \label{T:morrey-optimal-range}
  Assume that  $\Omega$, $m$,  and $\|\cdot\|_{X(0,1)}$ are as in Theorem~\ref{T:2}. Let  $\widehat\varphi\colon(0,\infty)\to(0,\infty)$ be the function defined as
      \begin{equation}\label{E:optimal-morrey-range}
        \widehat\varphi(r) =
            \|s^{-1+\frac{m}{n}}\chi_{(r^n,1)}(s)\|_{X'(0,1)} \quad  \ \text{if} \  r\in (0,\tfrac12],
    \end{equation}
    and               $ \widehat\varphi(r) = \widehat\varphi\left(\tfrac 12\right)$
if $r\in (\tfrac12,\infty)$.
Then, $\widehat\varphi$ is admissible and
    \begin{equation}\label{E:optimal-morrey-range-embedding}
        W^m X(\Omega) \to \mathcal M^{\widehat\varphi(\cdot)}(\Omega).
    \end{equation}
 Moreover, $\mathcal M^{\widehat\varphi(\cdot)}(\Omega)$ is the optimal (smallest possible) Morrey space in~\eqref{E:optimal-morrey-range-embedding}.
\end{thm}

\begin{thm}{\rm\bf{ [Optimal Sobolev domain for a Morrey target]}}\label{T:morrey-optimal-domain}
  Let   $\Omega$, $m$,    $\|\cdot\|_{X(0,1)}$ and $\varphi$ be as in Theorem~\ref{T:2}. Assume, in addition, that
    \begin{equation}\label{E:vanishing-condition}
               \inf_{r>0}\varphi(r)>0.
    \end{equation}
    Then, the functional $\|\cdot\|_{X^{\#}(0,1)}$, given by
    \begin{equation}\label{E:optimal-morrey-domain}
        \|f\|_{X^{\#}(0,1)} = \sup\limits_{r\in(0,1)} \, \frac{1}{\varphi(r^{\frac{1}{n}})}\int_{r}^{1}s^{-1+\frac{m}{n}}f^{**}(s)\,\d s
    \end{equation}
for $f\in L^0_+(0,1)$,
    is a rearrangement-invariant function norm and
    \begin{equation}\label{E:optimal-morrey-domain-embedding}
        W^m X^{\#}(\Omega) \to \mathcal M^{\varphi(\cdot)} (\Omega).
    \end{equation}
    Moreover, $X^{\#}(\Omega)$ is the optimal (largest possible) rearrangement-invariant space in~\eqref{E:optimal-morrey-domain-embedding}.
   \\ If condition~\eqref{E:vanishing-condition} is not fulfilled, then embedding~\eqref{E:SO-MAemb} fails for any rearrangement-invariant space $X(\Omega)$.
\end{thm}

We conclude our discussion about embeddings into Morrey type spaces by considering a variant where the space~$ \mathcal M^{\varphi(\cdot)} (\Omega)$ is replaced with its vanishing version.
The \emph{vanishing Morrey space}~$V \!\mathcal{M}^{\varphi(\cdot)}(\Omega)$ is defined as the collection of all functions $u\in \mathcal M^{\varphi(\cdot)}(\Omega)$ satisfying
\begin{equation}\label{E:morrey-vanishing-definition}
    \lim_{r\to 0^+}\psi_{\varphi, u}(r)=0,
\end{equation}
where
\begin{equation}\label{E:psi-definition}
    \psi_{\varphi, u}(r)=\sup_{B\subset\Omega,|B|\le r} \, \frac{1}{\varphi(|B|^{\frac1n})} \, \dashint_B|u|\,\d x \quad \ 
\hbox{for} \ r\in(0,1).
\end{equation}
The subject of the following result is a necessary and sufficient condition for embeddings of the form
 \begin{equation}\label{nov9}
        W^{m}X(\Omega) \to V\!\mathcal M^{\varphi(\cdot)}(\Omega).
    \end{equation}
 They have to be interpreted uniformly, in the sense that
     \begin{equation}\label{E:morrey-vanishing-full}
        \lim_{r\to 0^+} \, \sup\left\{\psi_{\varphi, u}(r)\colon \|u\|_{W^{m}X(\Omega)}\le 1\right\}=0.
    \end{equation}

\begin{thm}{\rm\bf{ [Criterion for embeddings into vanishing Morrey spaces]}}\label{T:sobolev-morrey-vanishing}
   Let  $\Omega$, $m$,  ${\|\cdot \|_{X(0,1)}}$ and $\varphi$ be as in Theorem~\ref{T:2}.
   Then,
    \begin{equation}\label{fi2}
        W^{m}X(\Omega) \to V\!\mathcal M^{\varphi(\cdot)}(\Omega)
    \end{equation}
    if and only if
    \begin{equation}\label{E:sobolev-to-morrey-vanishing-condition}
        \lim_{r\to 0^+}\frac{1}{\varphi( r)}\left\|s^{-1+\frac{m}{n}}\chi_{{(r^n,1)}}(s)\right\|_{{X'(0,1)}}=0.
    \end{equation}
 \end{thm}

\subsection{Embeddings into Campanato spaces} \label{S:campanato}

This subsection is devoted to embeddings of Sobolev spaces into generalized Campanato spaces.
The Campanato space~$\mathcal L^{\varphi(\cdot)}(\Omega)$ associated with an admissible function $\varphi$ consists of those functions $u\in L^1_{\rm loc}(\Omega)$ such that the semi-norm
\begin{equation}\label{E:campanato-semi}
    |u|_{\mathcal L^{\varphi(\cdot)}(\Omega)} = \sup_{B\subset \Omega} \, \frac{1}{\varphi(|B|^{\frac{1}{n}})}\,
 \dashint_{B}|u -u_{B}|\,\d x
\end{equation}
is finite. Here, $u_B$ denotes the average of $u$ over $B$.   Like for Morrey spaces, when $\varphi$ is a power $\alpha$ as in \eqref{power}, we adopt 
the notation $\mathcal L^{\alpha}(\Omega)$ for $\mathcal L^{\varphi(\cdot)}(\Omega)$. Of course, a shift of the exponent by $n$ with respect to the usual notation for classical Campanato occurs here as well, because of the integral average on the right-hand side of \eqref{E:campanato-semi}.

 Campanato spaces of higher-order $k \in \mathbb N$ are defined by measuring the integral deviation on balls of a function $u$ from a suitable polynomial of degree $k$, instead of just a constant. Of course, this requires a normalization by an extra $k$-th power of the radius of balls. Diverse definitions of these spaces, corresponding to different possible choices of the polynomials,  are available in the literature -- see e.g.~\cite{Samko} for a comprehensive survey on this topic. Since we are concerned with embeddings of higher-order Sobolev spaces, we can allow for polynomials, depending on each ball $B$, defined by prescribing that the averages on $B$ of all their derivatives up to the order $k$ agree with those of the corresponding derivatives of $u$.
Namely,
given  $k\in\mathbb N\cup\{0\}$, we  define the \emph{$k$-th order Campanato space} $\mathcal L^{k, \varphi(\cdot)}(\Omega)$ as the collection of all functions $u \in W^{k,1}(\Omega)$ such that the semi-norm
\begin{equation}\label{E:campanato-seminorm}
    |u|_{\mathcal L^{k, \varphi(\cdot)}(\Omega)} = \sup_{B\subset \Omega} \, \frac{1}{\varphi(|B|^{\frac{1}{n}})|B|^{\frac{k}{n}}} \, \dashint_{B}\left|u -{P}^{k}_{B}[u] \right|\,\d x
\end{equation}
is finite, where ${P}^{k}_{B}[u]$ denotes the polynomial of degree at most $k$ such that
\begin{equation}\label{E:moments}
    \int_{B}\nabla^{h}u \,\d x = \int_{B}\nabla^{h}{P}^{k}_{B}[u] \,\d x \quad \ 
\hbox{for} \ h=0,\dots,k.
\end{equation}
Here, $\nabla^{h}u$ denotes the vector of all weak derivatives of $u$ of order $h$, with the convention that $\nabla^0u =u$. The vector $\nabla ^1 u$ will also simply be denoted by $\nabla u$.
When $k=0$ this definition  recovers~\eqref{E:campanato-semi}.  Moreover, it is well suited for
the characterization of the Sobolev embeddings in question, which is the subject of the next result.

\begin{thm}{\rm\bf{ [Criterion for embeddings into Campanato spaces]}}\label{T:3}
 Assume that $\Omega$ is a bounded John domain in $\mathbb R^n$,
    $m\in \mathbb N$,  and $k\in\{0,\dots,m-1\}$. Let $\|\cdot \|_{X(0,1)}$ be a~re\-ar\-ran\-ge\-ment-invariant function norm, and let $\varphi$  be an admissible function. Then, the following statements are equivalent.
    \begin{enumerate}[(i)]
        \item The  embedding
        \begin{equation}\label{E:sobolev-to-campanato-embedding}
            W^{m}X(\Omega) \to \mathcal L^{k, \varphi(\cdot)}(\Omega)
        \end{equation}
        holds.
        \item The inequality
        \begin{equation}\label{E:sobolev-to-campanato-inequality}
            |u|_{\mathcal L^{k, \varphi(\cdot)}(\Omega)}  \leq c \sum_{j=k+1}^{m}\|\nabla^j u \|_{X(\Omega)}
        \end{equation}
    holds for some constant $c$ and for every  $u \in W^{m}X(\Omega)$.
        \item Either $k \in   \{0,\dots,m-2\}$ and
        \begin{equation}\label{E:sobolev-to-campanato-condition-subcritical}
            \sup_{r \in (0,1)} \, \frac{r}{\varphi(r)}\left\|s^{-1+\frac{m-(k+1)}{n}}\chi_{{(r^n,1)}}(s)\right\|_{{X'(0,1)}}<\infty,
        \end{equation}
        or $k=m-1$ and
        \begin{equation}\label{E:sobolev-to-campanato-condition-critical}
            \sup_{r\in(0,1)} \, \frac{1}{\varphi(r)r^{n-1}}\left\|\chi_{{(0,r^n)}}\right\|_{{X'(0,1)}}<\infty.
        \end{equation}
    \end{enumerate}
\end{thm}

\begin{remark}\label{rem-nov8}
    If  $\liminf_{r\to0^+} \, \frac{\varphi(r)}{r}=0
   $, then, owing to the admissibility of $\varphi$, one necessarily has that $\liminf_{r\to0^+}\varphi(r)=0$. Hence, neither of the conditions~\eqref{E:sobolev-to-campanato-condition-subcritical} and~\eqref{E:sobolev-to-campanato-condition-critical} can hold, whatever $\Omega$, $n$,  $m$,   and $\|\cdot\|_{X(0,1)}$ are. This is readily seen from the fact that both the functions $r\mapsto \left\|s^{-1+\frac{m-(k+1)}{n}}\chi_{{(r,1)}}(s)\right\|_{{X'(0,1)}}$ and $r\mapsto {r^{-1}}{\left\|\chi_{{(0,r)}}\right\|_{{X'(0,1)}}}$ are non-increasing on $(0,1)$. This corresponds to the fact that $\mathcal L^{k, \varphi(\cdot)}(\Omega)$ consists only of polynomials of degree not exceeding $k$ in these cases.
\end{remark}

\begin{remark}\label{rem-nov5}
    Like ~\eqref{E:7}, the conditions~\eqref{E:sobolev-to-campanato-condition-subcritical} and~\eqref{E:sobolev-to-campanato-condition-critical}
    are unchanged after replacing
    $\sup_{r \in (0,1)}$ with  $\sup_{r \in (0,r_0)}$, for any $r_0\in (0,1)$.
\end{remark}

\begin{remark}\label{rem-nov1}
    In case when $\Omega$ is a cube, $m=1$, and $k=0$, the equivalence of~\eqref{E:sobolev-to-campanato-embedding} and~\eqref{E:sobolev-to-campanato-condition-critical} recovers part of the content of~\cite[Theorem~1.1]{CPcamp}. The condition in that theorem is expressed as
    \begin{equation}\label{E:conditon-from-CP:2003}
            \sup_{r\in(0,1)} \, \frac{1}{\varphi(r)  \, r^n}\left\|s^{\frac1n}\chi_{(0,r^n)}(s)\right\|_{{X'(0,1)}}<\infty,
    \end{equation}
    which is equivalent to~\eqref{E:sobolev-to-campanato-condition-critical}, owing to Lemma~\ref{L:comparison-of-conditions} below.
\end{remark}

\begin{remark}\label{R:m-k}
    It follows from Theorem~\ref{T:3} that the embedding~\eqref{E:sobolev-to-campanato-embedding}  depends on $m$ and $k$ just through their difference $m-k$. In particular, an embedding
    \begin{equation*}
        W^{m}X(\Omega) \to \mathcal L^{k,\varphi(\cdot)}(\Omega)
    \end{equation*}
    holds for some 
   $m$, $k$, $X(\Omega)$ and $\varphi$
    if and only if
    \begin{equation*}
        W^{m-k}X(\Omega) \to \mathcal L^{\varphi(\cdot)}(\Omega).
    \end{equation*}
    Hence, if $u \in W^{m}X(\Omega)$ and the above embedding holds, then 
    \begin{equation}
    \label{fi1} \nabla^k u \in  \mathcal L^{\varphi(\cdot)}(\Omega).
    \end{equation}
    \end{remark}

\begin{thm}{\rm\bf{ [Optimal Campanato target]}} \label{T:campanato-optimal-range}
    Assume that $\Omega$, $m$, $k$, and $\|\cdot \|_{X(0,1)}$   are as in Theorem~\ref{T:3}. Let $\overline\varphi \colon (0,\infty) \to (0, \infty)$ be the function defined for $r\in \left(0,\tfrac12\right]$ as
        \begin{equation}\label{E:optimal-campanato-range-k}
        \overline\varphi(r) =
            r \, \|s^{-1+\frac{m-(k+1)}{n}}\chi_{(r^n,1)}(s)\|_{X'(0,1)} \quad \ 
\hbox{if} \ k\in\{0,\dots,m-2\},
    \end{equation}
and
    \begin{equation}\label{E:optimal-campanato-range-m-1}
        \overline\varphi(r) =
            r^{-n + 1}\, \|\chi_{(0,r^n)}\|_{X'(0,1)} \quad \ 
\hbox{if} \ k=m-1,
    \end{equation}
     and continued by   $ \overline\varphi(r) =  \overline\varphi(\tfrac 12)$ for
$r\in (\tfrac12,\infty)$.
    Then, $\overline\varphi$ is admissible and
    \begin{equation}\label{E:optimal-campanato-range-embedding}
        W^m X(\Omega) \to \mathcal L^{k, \overline\varphi(\cdot)}(\Omega).
    \end{equation}
 Moreover, $\mathcal L^{k, \overline\varphi(\cdot)}(\Omega)$ is the optimal (smallest possible) Campanato space of order $k$ in~\eqref{E:optimal-campanato-range-embedding}.
\end{thm}

\begin{thm}{\rm\bf{ [Optimal Sobolev domain for a Campanato target]}}\label{T:campanato-optimal-domain}
    Let  $\Omega$, $m$, $k$, and $\varphi$   be as in Theorem~\ref{T:3}. Assume that
   \begin{equation}\label{E:Campanato-condition-on-phi}  
    \liminf_{r\to0^+} \, \frac{\varphi(r)}{r}>0.
    \end{equation}
    Let $\|\cdot\|_{\widehat X(0,1)}$ be the functional defined as follows. If $m\le n+k$,  
    \begin{equation}\label{E:optimal-campanato-domain}
        \|f\|_{\widehat X(0,1)} =
        \begin{cases}
             \displaystyle  \sup\limits_{r\in(0,1)} \, \frac{r^{\frac{1}{n}}}{\varphi(r^{\frac{1}{n}})}\int_{r}^{1}s^{-1+\frac{m-(k+1)}{n}}f^{**}(s)\,\d s
             &\quad \text{if} \ k\in\{0,\dots,m-2\}
                \\[3ex]
            \sup\limits_{r\in(0,1)}\frac{r^{\frac{1}{n}}}{\varphi(r^{\frac{1}{n}})}f^{**}(r)
             &\quad \text{if} \ k=m-1
        \end{cases}
    \end{equation}
    for $f  \in L^0_+(0,1)$; if $m\ge n+k+1$,   $$\|f\|_{\widehat X(0,1)} =\|f\|_{L^1(0,1)}$$ 
    for $f  \in L^0_+(0,1)$. Then, $\|\cdot\|_{\widehat X(0,1)}$
    is a re\-ar\-ran\-ge\-ment-in\-va\-ri\-ant function norm and
    \begin{equation}\label{E:optimal-campanato-domain-embedding}
        W^m \widehat X(\Omega) \to \mathcal L^{k, \varphi(\cdot)}(\Omega).
    \end{equation}
   Moreover, $\widehat X(\Omega)$ is the optimal (largest possible) rearrangement-invariant space in~\eqref{E:optimal-campanato-domain-embedding}.\\
   If the condition~\eqref{E:Campanato-condition-on-phi} is not satisfied, then the embedding~\eqref{E:sobolev-to-campanato-embedding} fails for any rearrangement-invariant space $X(\Omega)$.
\end{thm}

\begin{ex}\label{BMOex}{\rm As a consequence of Theorem~\ref{T:campanato-optimal-domain}, the optimal domain for embeddings into the space $\BMO(\Omega)$,
   for $m\in\mathbb N$,  with $m\le n$, can be characterized in any  bounded John domain  $\Omega\subset \mathbb R^n$. Indeed, such an optimal domain is provided by Proposition~\ref{P:optimal-domain-for-BMO}, Section~\ref{special}, which tells us that
    \begin{equation*}\label{E:orlicz-campanato-optimal-range}
        W^{m}L^{\frac nm,\infty}(\Omega)\to\BMO(\Omega).
    \end{equation*}}
    \end{ex}

Given an admissible function $\varphi$, the \emph{vanishing Campanato space}~$V\!\mathcal{L}^{k,\varphi(\cdot)}(\Omega)$ stands to the Campanato space
$\mathcal L^{k, \varphi(\cdot)}(\Omega)$ as the Sarason space ${\rm VMO}(\Omega)$ of functions of vanishing means oscillation stands to ${\rm BMO}(\Omega)$.
The space $V\!\mathcal{L}^{k,\varphi(\cdot)}(\Omega)$ is defined as the collection of all functions $u\in\mathcal L^{k, \varphi(\cdot)}(\Omega)$ satisfying
\begin{equation}\label{E:vanishing-definition}
    \lim_{r\to 0^+}\varrho_{\varphi,k,u}(r)=0,
\end{equation}
where
\begin{equation}\label{E:varrho-definition}
    \varrho_{\varphi,k,u}(r)=\sup_{B\subset\Omega,|B|\le r} \, \frac{1}{\varphi(|B|^{\frac1n})|B|^{\frac{k}{n}}} \,\dashint_B \left|u -{P}^{k}_{B}[u] \right|\,\d x \quad \ \text{for} \ r\in(0,1).
\end{equation}

Our last main result provides us with a characterization of Sobolev embeddings of the form
   \begin{equation*}
        W^{m}X(\Omega) \to V\!\mathcal{L}^{k,\varphi(\cdot)}(\Omega).
    \end{equation*}
   In analogy with~\eqref{nov9}, the latter embedding means that
\begin{equation}\label{E:vanishing-full}
        \lim_{r\to 0^+} \, \sup\left\{\varrho_{\varphi, k, u}(r)\colon \|u\|_{W^{m}X(\Omega)}\le 1 \right\}=0.
    \end{equation}

\begin{thm}{\rm\bf{ [Criterion for embeddings into vanishing Campanato spaces]}}\label{T:sobolev-campanato-vanishing}
    Let  $\Omega$, $m$, $k$,   ${\|\cdot \|_{X(0,1)}}$, and $\varphi$    be as in Theorem~\ref{T:3}.
    Then, the following statements are equivalent:
    \begin{enumerate}[(i)]
    \item The embedding
    \begin{equation}\label{E:sobolev-to-campanato-vanishing-embedding}
        W^{m}X(\Omega) \to V \!\mathcal L^{k, \varphi(\cdot)}(\Omega)
    \end{equation}
    holds.

    \item
       One has
    \begin{equation}\label{E:vanishing-nonfull}
        \lim_{r\to 0^+} \, \sup\left\{\varrho_{\varphi, k, u}(r)\colon \sum_{j=k+1}^m\|\nabla^j u\|_{X(\Omega)}\le 1\right\}=0.
    \end{equation}

    \item Either $k\in  \{0,\dots,m-2\}$ and
    \begin{equation}\label{E:sobolev-to-campanato-vanishing-condition-subcritical}
        \lim_{r\to 0^+}\frac{r}{\varphi(r)}\left\|s^{-1+\frac{m-(k+1)}{n}}\chi_{{(r^n,1)}}(s)\right\|_{{X'(0,1)}}=0,
    \end{equation}
    or $k=m-1$ and
    \begin{equation}\label{E:sobolev-to-campanato-vanishing-condition-critical}
        \lim_{r\to 0^+} \, \frac{1}{\varphi(r)r^{n-1}}\left\|\chi_{{(0,r^n)}}\right\|_{{X'(0,1)}}=0.
    \end{equation}
    \end{enumerate}
\end{thm}

\subsection{Comparison between embeddings into Campanato and H\"older type spaces}\label{S:comp}
 It can be interesting to compare the Sobolev embeddings into Campanato spaces discussed in the present paper with parallel embeddings into H\"older type spaces of uniformly continuous functions, which were established in~\cite{CPcamp} and~\cite{Monia1} for first and higher order Sobolev spaces, respectively.
\\ Let $\sigma \colon (0, \infty) \to (0, \infty)$ be a modulus of continuity, namely a function equivalent, up to positive multiplicative constants, to a 
 non-decreasing function   and such that
    \begin{equation}\label{feb37}
        \lim_{r\to0^+}\sigma(r)=0\qquad \text{and} \qquad \liminf_{r\to0^+}\frac {\sigma(r)}r>0.
    \end{equation}
     The   space $C^{\sigma(\cdot)}(\Omega)$ is defined as the collection of all functions $u\colon\Omega\to\mathbb R$ for which the semi-norm
    \begin{equation*}
        |u|_{C^{\sigma(\cdot)}(\Omega)} = \sup_{x,y\in\Omega,\,x\neq y}\frac{|u(x)-u(y)|}{\sigma(|x-y|)} 
    \end{equation*} is finite.
    \\ A combination of the results from~\cite{CPcamp} and~\cite{Monia1} tells us the following. Define the function $\vartheta\colon (0, \infty) \to (0, \infty]$, for $1\leq m \leq n-1$, as
    \begin{equation*}\label{theta}
        \vartheta (r) =  \|s^{-1+\frac{m}{n}}\chi_{(0, r^n)}(s)\|_{X'(0,1)} \quad \ \text{for} \  r\in (0,1),
    \end{equation*}
    and the function 
    $\varrho\colon (0, \infty) \to (0, \infty]$, for $2\leq m \leq n$, as
    \begin{equation*}\label{rho}
        \varrho (r) = r\|s^{-1+\frac{m-1}{n}}\chi_{(r^n,1)}(s)\|_{X'(0,1)} \quad \ \text{for} \  r\in (0,1).
    \end{equation*}
Then, under suitable regularity assumptions on $\Omega$, an embedding of the form
\begin{equation}\label{cont}    W^mX(\Omega)\to C^{\sigma(\cdot)}(\Omega)\quad
    \end{equation}
holds if and only if the function $\widehat \sigma$, given by
\begin{align*}\label{feb31}
    \widehat \sigma (r) = \begin{cases}
        \vartheta (r) &\quad \text{if} \ m=1
        \\ \vartheta (r) + \varrho (r) &\quad \text{if} \ 2\leq m \leq n-1
        \\ \varrho (r) &\quad \text{if} \ m=n
    \end{cases}
\end{align*}
is finite-valued, and 
\begin{equation*}
    \label{feb30}
    \lim_{r\to 0^+}\widehat\sigma (r) = 0.
\end{equation*}
Moreover, $\widehat \sigma$ is the optimal modulus of continuity $\sigma$ in \eqref{cont}.
\\
In the realm of the standard Sobolev spaces $W^{m.p}(\Omega)$, the optimal Campanato and H\"older target spaces are the classical ones associated with power type functions. Moreover, the powers of the two spaces agree.
In fact, one proof of the Morrey embedding
\begin{equation}
    \label{morrey}
    W^{1,p}(\Omega) \to
   C^{{1-\frac np}}(\Omega)
   \end{equation}
for $p>n$ combines an argument showing that
\begin{equation*}
    \label{feb32a}
    W^{1,p}(\Omega) \to 
   \mathcal L^{{1-\frac np}}(\Omega)
 \end{equation*}
with the equality
\begin{equation*}
    \label{feb33}
   C^{{1-\frac np}}(\Omega)
    =\mathcal L^{{1-\frac np}}(\Omega)
\end{equation*}
up to equivalent norms. More generally, the embedding
\begin{equation*}
    \label{feb34}
    W^{m,p}(\Omega) \to 
 C^{{\alpha}}(\Omega)
\end{equation*}
for some $m\in \mathbb N$, $p \geq 1$ and $\alpha >0$ holds if and only if 
the embedding
\begin{align}
    \label{feb32}
    W^{m,p}(\Omega) \to
 \mathcal L^{{\alpha}}(\Omega)
   \end{align}
holds.
\\ On the contrary, an embedding of the form 
\begin{equation*}
    \label{feb36}
    W^{m}X(\Omega) \to \mathcal L^{\sigma (\cdot)}(\Omega),
\end{equation*}
 need not imply the companion embedding
\begin{equation*}
    \label{feb38}
    W^{m}X(\Omega) \to C^{\sigma(\cdot)}(\Omega)  
\end{equation*}
in the broader class of all rearrangement-invariant spaces $X(\Omega)$ and functions
 $\sigma$ fulfilling the conditions~\eqref{feb37}. This is related to the fact that, whereas
\begin{equation*}
    \label{feb39}
    C^{\sigma(\cdot)}(\Omega)   \to  \mathcal L^{\sigma (\cdot)}(\Omega)
\end{equation*} 
for every modulus of continuity $\sigma$, the converse embedding holds if and (under some mild qualification) only if
\begin{equation*}\label{feb41}
    \int_0^r \frac{\sigma (s)}{s}\, \d s \ \approx \ \sigma (r) \quad \ \text{near} \ 0.
\end{equation*}
The latter assertion is a consequence of a result from~\cite{Sp}, which tells us that the embedding 
\begin{align}
    \label{feb40}
    \mathcal L^{\varphi (\cdot)}(\Omega) \to C^{\sigma(\cdot)}(\Omega)
\end{align}
holds for some admissible $\varphi$ if and (under the same qualification) only if 
\begin{equation*}
    \label{feb42}
    \int_0^r \frac{\varphi (s)}{s}\, \d s \ \lesssim \ \sigma (r) \quad \ \text{near} \ 0.
\end{equation*}
For instance, as shown in Corollary \ref{T:example-zygmund} below, if $1\leq m \leq n$ and $\alpha>0$, then
\begin{equation}
    \label{feb45}
    W^{m}L^{\frac nm} \left(\log L\right)^{\alpha}(\Omega) \to \mathcal L^{\varphi (\cdot)}(\Omega),
\end{equation}
where
\begin{equation*}
    \label{feb46}
    \varphi(r)   \approx  \left(\log \tfrac1r\right)^{-\frac{\alpha m}{n}}\quad \ \text{near} \ 0.
\end{equation*}
On the other hand, the space $ W^{m}L^{\frac nm}  \left(\log L\right)^{\alpha}(\Omega)$ is embedded into a space of   H\"older  continuous functions only if $$\alpha > \frac{n-m}{m},$$ and 
\begin{equation}\label{E:eleven}
        W^mL^{\frac{n}{m}}\left(\log L\right)^{\alpha}(\Omega) \to C^{\sigma(\cdot)}(\Omega)
    \end{equation}
    with    \begin{equation*}\label{E:twelve}
        \sigma(r) \   \approx \ 
        \begin{cases}
            \left(\log \tfrac1r\right)^{1-\frac{\alpha+1}{n}}&\quad \text{if} \  m=1
                \\[0.5mm]
            \left(\log \tfrac1r\right)^{1-\frac{(\alpha+1)m}{n}}&\quad \text{if} \  2\le m\le n-1 
                \\[1mm]
            \left(\log \tfrac1r\right)^{-\alpha}&\quad \text{if} \ m=n
       \end{cases}
    \end{equation*}
near $0$. 
\\ Let us also stress that the two-step argument via~\eqref{feb32} to derive the embedding~\eqref{morrey} need not yield sharp results in the general framework under consideration. Indeed, a combination of~\eqref{feb45} with~\eqref{feb40} is only possible if 
$$\alpha > \frac nm,$$
and results in a weaker embedding of the form~\eqref{E:eleven}, with
\begin{equation}\label{E:ten}
        \sigma(r) \   \approx \  \left(\log \tfrac1r\right)^{1-\frac{\alpha m}{n}} \quad \ \text{near} \ 0.
    \end{equation}

\subsection{A special instance: embeddings of Orlicz-Sobolev spaces}
\label{S:ex}

The results from the previous section can be implemented  in the case of Orlicz--Sobolev domain spaces to derive the optimal embeddings into Morrey and Campanato spaces presented below.   Owing to Remark~\ref{rem-trivialmorrey}, when dealing with Morrey spaces, we shall assume that $m\leq n-1$.

 \begin{thm}{\rm\bf{ [Optimal Morrey target for an Orlicz-Sobolev domain]}} \label{T:orlicz-morrey-optimal-range}
    Let $\Omega$ and $m$ be as in Theorem~\ref{T:2}, and let~$A$ be a Young function.
    Define the function $E_m \colon [0, \infty ) \to [0, \infty ]$ by
    \begin{equation}\label{E:young-for-orlicz-optimal-morrey-range}
        E_m(t) =
             t^{ \frac{n}{n-m}} \int_{0}^{t} \ \frac{\widetilde A(\tau)}{\tau^{1+ \frac{n}{n-m}}}\,\d \tau\quad \  \text{for} \ t\in[0,\infty),
    \end{equation}
    and let $\widehat\varphi_A \colon (0,\infty) \to (0, \infty)$ be the function defined    as
    \begin{equation}\label{E:weight-for-orlicz-optimal-morrey-range}
         \widehat\varphi_A(r)  =\frac{1}{r^{n-m}E_m^{-1}\left(r^{-n}\right)} \quad \ \mbox{if} \ r\in(0,1],
    \end{equation}
     and $\widehat\varphi_A(r)=\widehat\varphi_A(1)$ if $r\in(1,\infty)$. Then, $\widehat\varphi_A$ is an admissible function, and
    \begin{equation}\label{E:Orlicz-optimal-morrey-range-embedding}
        W^{m, A}(\Omega) \to \mathcal M^{\widehat\varphi_A(\cdot)}(\Omega).
    \end{equation}
    Moreover, $\mathcal M^{\widehat\varphi_A(\cdot)}(\Omega)$ is the optimal (smallest possible) Morrey target space in~\eqref{E:Orlicz-optimal-morrey-range-embedding}.
\end{thm}

An application of Theorem~\ref{T:orlicz-morrey-optimal-range} to Orlicz-Sobolev spaces built upon Young functions $A$ such that
$$
A(t) \ \simeq \ t^p(\log t)^\alpha\qquad \text{near infinity},$$
provides us with the following embeddings.   In the corresponding statements, we shall adopt the convention that $\frac 10=\infty$.

\begin{corollary}{\rm\bf{ [Optimal Morrey target for a Zygmund-Sobolev domain]}} \label{C:example-zygmund-morrey}
       Let $\Omega$ and $m$ be as in Theorem~\ref{T:2}. Assume that  either $p\in(1,\infty)$ and $\alpha\in\mathbb R$, or $p=1$ and $\alpha \geq 0$.
     Let  $\widehat\varphi\colon(0,\infty)\to(0,\infty)$ be an admissible function such that:
\begin{align*}
   \widehat\varphi(r)  \approx 
        \begin{cases}
            r^{m-\frac{n}{p}}\left(\log\tfrac 1r\right)^{-\frac{\alpha}{p}}&\quad \text{if $p<\frac{n}{m}$}
               \\[2ex]
            \left(\log\tfrac 1r\right)^{1-\frac{(\alpha+1)m}{n}}&\quad \text{if $p=\frac{n}{m}$ and $\alpha<\frac{n-m}{m}$}
              \\[2ex]
            \left(\log\log\tfrac 1r\right)^{1-\frac{m}{n}}&\quad \text{if   $p=\frac{n}{m}$ and $\alpha=\frac{n-m}{m}$}
             \\[2ex]
             1&\quad \text{if either  $p=\frac{n}{m}$ and $\alpha>\frac{n-m}{m}$, or $p>\frac{n}{m}$,}                \end{cases} 
\end{align*}
near $0$.
\\ Then, 
    \begin{equation}\label{E:log-optimal-morrey-range-embedding}
        W^{m}L^{p}   \left(\log L\right)^{\alpha}(\Omega) \to \mathcal M^{\widehat\varphi(\cdot)}(\Omega),
    \end{equation}
    and  the target space in~\eqref{E:log-optimal-morrey-range-embedding} is optimal (smallest possible) among all Morrey spaces.
\end{corollary}

A further particular case of Corollary~\ref{C:example-zygmund-morrey} is contained in  the next result.

\begin{corollary}{\rm\bf{ [Optimal Morrey target for a classical Sobolev domain]}}  \label{T:example-lebesgue-morrey}
   Let $\Omega$ and $m$ be as in Theorem~\ref{T:2}.  Assume that  $p \in [1, \infty)$.
     Let  $\widehat\varphi\colon(0,\infty)\to(0,\infty)$ be an admissible function such that:
        \begin{equation*}\label{E:weight-for-lrbrdgur-optimal-morrey-range}
      \widehat\varphi(r) \ \approx \  \begin{cases}
            r^{m-\frac{n}{p}}&\quad \text{if $p<\frac{n}{m}$}
               \\[2ex]
            \left(\log\tfrac 1r\right)^{1-\frac{m}{n}}&\quad \text{if $p=\frac{n}{m}$}
                \\[2ex]
             1&\quad \text{if  $p>\frac{n}{m}$}
         \end{cases}
    \end{equation*}
  near $0$.

\noindent Then,
    \begin{equation}\label{E:leb-optimal-morrey-range-embedding}
            W^{m, p}(\Omega) \to   \mathcal M^{\widehat\varphi(\cdot)}(\Omega),
    \end{equation}
    and the target space in~\eqref{E:leb-optimal-morrey-range-embedding} is optimal (smallest possible) among all Morrey spaces.
\end{corollary}

\begin{thm}{\rm\bf{ [Optimal Campanato target for an Orlicz-Sobolev domain]}} \label{P:orlicz-campanato-optimal-range}
    Let $\Omega$, $m$, and $k$  be as in Theorem~\ref{T:3}, and let $A$ be a Young function.
    For $k \in \{0,\dots,m-2\}$, define the Young function $E_{m,k} \colon [0, \infty ) \to [0, \infty ]$ as in~\eqref{E:young-for-orlicz-optimal-morrey-range}, with $m$ replaced with $m-k-1$, namely 
    \begin{equation}\label{E:Young for Orlicz-optimal-campanato-range}
        E_{m,k}(t) =
             t^{ \frac{n}{n-m+k+1}} \int_{0}^{t} \ \frac{\widetilde A(\tau)}{\tau^{1+ \frac{n}{n-m+k+1}}}\,\d \tau\quad \ \text{for} \ t \in  [0, \infty),
    \end{equation}
and let  $\overline \varphi_A \colon (0, \infty) \to (0, \infty)$ be the function defined   as
    \begin{equation}\label{E:weight for Orlicz-optimal-campanato-range}
         \overline \varphi_A(r) =
        \begin{cases}  \displaystyle  \frac 1 {r^{m-k-n}E_{m,k}^{-1}\left(r^{-n}\right)}
             &\quad \text{if} \ k\in\{0,\dots,m-2\}
                \\[3ex]
                   r A^{-1}\left(r^{-n}\right) 
                           &\quad \text{if} \ k=m-1
        \end{cases}
    \end{equation}  for $r\in(0,1]$ 
     and   $\overline \varphi_A(r)=\overline \varphi(1)$ for $r\in(1,\infty)$. Then, $\overline \varphi_A$ is an admissible function, and
    \begin{equation}\label{E:Orlicz-optimal-campanato-range-embedding}
        W^{m, A}(\Omega) \to \mathcal L^{k, \overline \varphi_A(\cdot)}(\Omega).
    \end{equation}
    Moreover, $\mathcal L^{k, \overline \varphi_A(\cdot)}(\Omega)$ is the optimal (smallest possible) $k$-th order Campanato target space in~\eqref{E:Orlicz-optimal-campanato-range-embedding}.
\end{thm}

\begin{corollary}{\rm\bf{ [Optimal Campanato target for a Zygmund-Sobolev domain]}} \label{T:example-zygmund}
     Let $\Omega$, $m$, $k$   be as in Theorem~\ref{T:3}. Assume that  either $p\in(1,\infty)$ and $\alpha\in\mathbb R$, or $p=1$ and $\alpha \geq 0$.    Let  $\overline \varphi \colon(0,\infty)\to(0,\infty)$ be an admissible function such that: 
    \begin{align*}
   &
  \overline \varphi(r) \ \approx \
        \begin{cases}
         r^{ m-k- \frac{n}{p} }\left(\log \frac 1r\right)^{ -  \frac{\alpha}{p} } &\quad \text{if} \ p <\frac{n}{m-k-1} \ \mbox{and} \  \alpha \in \mathbb R
             \\[3ex]
             r \left(\log \frac 1r\right)^{1-  \frac{(m-k-1)(1+\alpha)}{n} } &\quad \text{if} \  p =\frac{n}{m-k-1} \ \mbox{and} \ \alpha< \frac{n-m+k+1}{m-k-1}
          \\[4ex]
             r \left(\log\log \frac 1r\right)^{1-\frac{m-k-1}{n} } &\quad \text{if} \   p =\frac{n}{m-k-1} \ \mbox{and} \ \alpha= \frac{n-m+k+1}{m-k-1}
               \\[3ex]
             r  &\quad \text{if} \  p =\frac{n}{m-k-1}\ \mbox{and} \ \alpha> \frac{n-m+k+1}{m-k-1}
                \\[1ex]
            &\quad     \ \text{or} \  p >\frac{n}{m-k-1}  \ \mbox{and} \ \alpha\in \mathbb R
        \end{cases}
    \qquad\quad \text{for $1 \leq m -k\le n$,}
\\ \\  
   &    \overline \varphi(r) \ \approx \ r  
   \qquad \qquad\qquad\qquad \qquad\qquad\qquad \qquad\qquad\qquad  \qquad\qquad\qquad\quad \ \  \quad   
   \text{for $m - k \geq n +1$,}
    \end{align*} 
near $0$.

\noindent Then,
    \begin{equation}\label{E:Log-optimal-campanato-range-embedding}
        W^{m}L^{p}   \left(\log L\right)^{\alpha}(\Omega) \to \mathcal L^{k, \overline \varphi(\cdot)}(\Omega),
    \end{equation}
    and  the target space in~\eqref{E:Log-optimal-campanato-range-embedding} is optimal (smallest possible) among all Campanato spaces of order~$k$.
\end{corollary}

\begin{corollary}{\rm\bf{ [Optimal Campanato target for a classical Sobolev domain]}}  \label{T:example-lebesgue}
     Let $\Omega$, $m$, $k$  be as in Theorem~\ref{T:3}.   Assume that $p \in [1, \infty)$.  
      Let  $\overline \varphi \colon(0,\infty)\to(0,\infty)$ be an admissible function such that: 
    \begin{align*} 
   &
  \overline \varphi(r) \ \approx \
        \begin{cases}
             r^{ m-k- \frac{n}{p} } &\quad \text{if} \ p <\frac{n}{m-k-1} 
                \\[2ex]
             r\left(\log \frac 1r\right)^{1-  \frac{m-k-1}{n} } &\quad \text{if} \ 
             p =\frac{n}{m-k-1} 
             \\[2ex]
             r &\quad \text{if} \ p >\frac{n}{m-k-1} 
        \end{cases}
    \qquad\quad \text{for $1 \leq m -k\le n$,}
\\ \\  
   &    \overline \varphi(r) \ \approx \ r  
   \qquad \qquad\qquad\qquad \qquad   \qquad\qquad\qquad\quad  \  \quad   
   \text{for $m - k \geq n +1$,}
    \end{align*} 
near $0$.

\noindent Then,
    \begin{equation}\label{E:Leb-optimal-campanato-range-embedding}
            W^{m, p}(\Omega) \to \mathcal L^{k, \overline \varphi(\cdot)}(\Omega),
    \end{equation}
and the target space in~\eqref{E:Leb-optimal-campanato-range-embedding} is optimal (smallest possible) among all  $k$-th order Campanato spaces.
\end{corollary}

    The results displayed above disclose that
    the use of wider families of Sobolev, Campanato and Morrey spaces ensures the existence of optimal embeddings in situations where they do not exist in the classical versions of the relevant spaces. It also permits a finer description of certain phenomena in borderline situations.
    \\ For instance, if $1\leq m \leq n$, then
$$W^{m}L^{\frac nm} (\Omega) \to \mathcal L^0(\Omega) = {\rm BMO}(\Omega),$$
and the target space is optimal among all Campanato type spaces. On the other hand, an optimal Morrey target space $\mathcal M^\alpha(\Omega)$ in the embedding
    $$W^{m}L^{\frac nm} (\Omega) \to \mathcal M^\alpha(\Omega)$$
    only exists, and agrees with $\mathcal M^0(\Omega)=L^\infty(\Omega)$, if $m=n$. Indeed, if $m<n$, such an embedding
 holds for every $\alpha <0$, but fails for every $\alpha \geq 0$.
\\ By contrast, one has that
\begin{align}\label{march8}
    W^{m}L^{\frac nm} (\Omega) \to \mathcal M^{\widehat \varphi(\cdot)}(\Omega),
    \end{align}
with 
$$\widehat \varphi (r) \approx \begin{cases}
    \left(\log\tfrac1r\right)^{1-\frac{m}{n}} & \quad \text{if $m<n$}
\\ 1  & \quad \text{if $m=n$,}
\end{cases}$$
near $0$,
the target space being optimal in \eqref{march8} among all Morrey type spaces.
\\ 
Optimal Campanato and Morrey target spaces also exist, and involve a richer variety of targets, for the borderline Zygmund-Sobolev spaces $W^{m}L^{\frac nm}   \left(\log L\right)^{\alpha}(\Omega)$, for any $\alpha \in \mathbb R$.
Indeed,
        \begin{align}\label{march1}
             W^{m}L^{\frac nm}   \left(\log L\right)^{\alpha}(\Omega) \to \mathcal L^{\overline \varphi(\cdot)}(\Omega)
        \end{align}
        where 
\begin{align}\label{march2}
\overline{\varphi}(r) \  \approx  \ \left(\log\tfrac1r\right)^{-\frac{\alpha m}{n}} \ \quad \mbox{near $0$.}
\end{align}
        Moreover, 
        \begin{align}\label{march3}
             W^{m}L^{\frac nm}   \left(\log L\right)^{\alpha}(\Omega) \to \mathcal M^{\widehat \varphi(\cdot)}(\Omega)
        \end{align}
        where
        \begin{align}\label{march4}
\widehat{\varphi}(r) \  \approx \ \begin{cases}
\left(\log\tfrac1r\right)^{1-\frac{(\alpha +1) m}{n}} & \quad \text{if $\alpha < \frac nm -1$}
\\ \big(\log\big(\log\tfrac1r\big)\big)^{1-\frac mn} & \quad \text{if $\alpha = \frac nm -1$}
\\ 1 & \quad \text{if $\alpha > \frac nm -1$}
\end{cases}
\end{align}  
near $0$,
and $$\mathcal L^{\overline \varphi(\cdot)}(\Omega) \neq \mathcal M^{\widehat \varphi(\cdot)}(\Omega)$$ in all cases.

\section{Background}\label{sec2}

Throughout the paper, the relation \lq \lq$\lesssim$” between two positive expressions means that the former is bounded by the latter, up to a multiplicative constant depending on quantities to be specified. The relations \lq \lq$\gtrsim$” and \lq \lq$\approx$” are defined accordingly.

\subsection{Rearrangement-invariant spaces}\label{sec2.2}

We recall some definitions and basic properties  of
decreasing rearrangements and rearrangement-invariant function
norms. For more details and proofs,  we refer to~\cite{BS, Lubook}.

Let $E$ be a Lebesgue measurable subset of $\mathbb R^n$, with $n \geq 1$. We denote by
$\chi_{_E}$  the characteristic function of $E$, and by $|E|$  its
Lebesgue measure.
\\
The notation $L^0(E)$ stands for the
Riesz space of  all (equivalence classes  of)  measurable functions from $E$ into $[-\infty , \infty]$  and
$L^0_+(E)= \{u \in L^0(E) \colon u \geq 0 \ \hbox{a.e. in} \, E\}$.

The \emph{decreasing rearrangement}~$u^*\colon [0, \infty) \rightarrow [0, \infty]$ of a function $u\in
L^0(E)$  is defined by
\begin{equation}\label{decreasing rearrangementE}
    u^*(r)= \inf\left\{s \geq 0\colon  \left|\{x\in E\colon |u(x)|>s\}\right|\leq r\right\} \quad \ \text{for} \ r\in[0,\infty).
\end{equation}
Clearly, $u^*(r) =0$ if $r \geq |E|$.

The \emph{Hardy-Littlewood inequality}~tells us that
\begin{equation}\label{HL.0}
    \int_{E}|u v |\, \d x\leq \int_{0}^{|E|}u^*v^*\, \d r
\end{equation}
for  $u, v\in L^0(E)$.

The function $u^{**}\colon (0,\infty)\to[0,\infty ]$   is defined as
\begin{equation}\label{u^**}
    u^{\ast \ast}(r)\, = \, {\frac 1 r}\, \int_0^{r}u^*  \, \d s \quad \ \text{for} \ r\in(0,\infty).
\end{equation}

Both functions $u^*$ and $u^{**}$ are non-increasing, and $u^* \leq u^{**}$ on $(0,\infty)$. Moreover,
\begin{equation}\label{**}
    (u+v)^{\ast \ast}  \, \leq \, u^{\ast \ast}  + v^{\ast \ast}
\end{equation}
for every $u,v \in L^0_+(E)$.

A \emph{rearrangement-invariant Banach (extended) function norm} -- an \emph{r.i.~function norm}, for short -- is a functional $\|\cdot\|_{{X(0,1)}}\colon L^0_+(0,1) \rightarrow[0,\infty]$ such that
\begin{align}
    \label{N1a}  &\|f+g\|_{{X(0,1)}} \,  \leq \, \|f\|_{{X(0,1)}} + \|g\|_{{X(0,1)}}   \quad  \text{for all}\  f,g \in L^0_+(0,1); \qquad \qquad \qquad \qquad \qquad   \hfill{}
        \\
    \label{N1b}  & \|\lambda f\|_{{X(0,1)}} \,     = \,   \lambda  \, \|f\|_{{X(0,1)}} \quad   \text{for all} \
    \lambda \geq 0,\,  f \in L^0_+(0,1);
        \\
    \label{N1c}  & \|f\|_{{X(0,1)}} \,   >\, 0  \quad  \text{if} \  f  \, \text{does
    not vanish a.e. in $(0,1)$;}
        \\
    \label{N2} &\|f\|_{{X(0,1)}} \, \leq \, \|g\|_{{X(0,1)}}  \  \text{whenever} \     0 \leq f \, \leq \, g  \  \text{a.e. in}  \ (0,1);
        \\
    \label{N3} & \sup_{k}\|f_k\|_{{X(0,1)}} \, = \, \| f\|_{{X(0,1)}}  \    \text{if}  \ \{f_k\} \subset L^0_{+}(0,1)    \    \text{with}  \    f_k \nearrow f  \  \text{a.e. in}  \ (0,1);
        \\
    \label{N4} &   \|1\|_{{X(0,1)}} \, < \, \infty ;
        \\
    \label{N5} &     \text{there exists a   constant $c$ such that} \ \int_0^1 f \, \d r  \leq c \, \|f\|_{{X(0,1)}}   \ \text{for all}\
    f  \in L^0_{+}(0,1) ;
        \\
    \label{N6} & \|f\|_{{X(0,1)}} \, = \, \| g\|_{{X(0,1)}}     \ \text{for all}\ f, g \in   L^0_{+}(0,1)    \    \text{such that}  \  f^\ast = g^\ast.
\end{align}

The \emph{associate function norm of}~$\|\cdot\|_{{X(0,1)}}$ is the r.i.~function norm $\|\cdot\|_{{X'(0,1)}}$ defined by
\begin{equation} \label{n.assoc.}
    \| g\|_{{X'(0,1)}}=  \sup   \left\{ \int_0^1 fg\,\d r \colon f \in
    L^0_+(0,1), \, \| f \|_{{X(0,1)}} \leq 1 \right\} 
\end{equation}   for $g \in
    L^0_+(0,1)$.
It follows from~\eqref{HL.0} that
\begin{equation}\label{n.assoc.2}
    \| g\|_{{X'(0,1)}}=  \sup   \left\{ \int_0^1 f^\ast g^\ast \, \d r \colon f \in
    L^0_+(0,1), \, \| f \|_{{X(0,1)}} \leq 1 \right\}  
\end{equation} for $g \in
    L^0_+(0,1)$.
The \emph{fundamental function} of a Banach function norm $\|\cdot\|_{{X(0,1)}}$ is defined as 
$$[0,\infty) \ni r \mapsto \|\chi_{(0,r)}\|_{{X(0,1)}}.$$
One has that
\begin{align}
    \label{feb100}
\|\chi_{(0,r)}\|_{{X(0,1)}} \,\|\chi_{(0,r)}\|_{{X'(0,1)}} = r \qquad \text{for $r\in (0,1)$.}
\end{align}

Given a~measurable set $E$ of finite positive measure and an r.i.~function norm
 $\|\cdot\|_{{X(0,1)}}$, the space $X(E)$ is
defined as the collection of all  functions  $f \in L^0(E)$ such
that the quantity
\begin{equation}\label{norm}
    \|f\|_{{X(E)}}=\|f\sp*(|E| \,\cdot) \|_{{X(0,1)}}
\end{equation}
is finite. This quantity defines a norm on $X(E)$, called an \textit{r.i.~norm}, which makes $X(E)$ a Banach space, called an
\textit{r.i.~space}. Note that the quantity $\|f\|_{{X(E)}}$ is defined for every $f\in L^{0}(E)$, and it is finite if and only if $f\in X(E)$.

The space $X(0,1)$ is called the \textit{representation space} of $X(E)$. The r.i.~space $X'(E)$ built upon the function norm $\|\cdot \|_{{X'(0,1)}}$ is called the \textit{associate space} of $X(E)$. One always has that $X(E)= X''(E)$.

The \emph{Hardy--Littlewood--P\'olya principle} tells us that if $f \in L^0(E)$ and $g \in X(E)$ are such that
\begin{equation} \label{Hardy}
    f^{**}(r) \leq g^{**}(r)\quad \ \text{for} \ r\in(0,\infty),
\end{equation}
then $f\in X(E)$ and $\|f\|_{{X(E)}}\leq\|g\|_{{X(E)}}$.

Given any $\lambda>0$, the \textit{dilation operator} $\mathcal E_{\lambda}$, defined at $f\in L^0(0,1)$ by
\begin{equation}\label{E:dilation}
    (\mathcal E_{\lambda}f)(r)=
    \begin{cases}
      f(\lambda\sp{-1}r)&\quad  \text{if} \  0<r\leq \lambda
        \\[0.5ex]
      0&\quad  \text{if} \ \lambda<r<1 
  \end{cases}
\end{equation}
is bounded on every rearrangement-invariant~space $X(0,1)$, with norm not exceeding $\max\{1, \frac 1{\lambda}\}$.

The \emph{H\"older inequality} tells us that
\begin{equation}\label{E:H}
    \int_E|fg| \, \d x \leq  \| f\|_{X(E)} \|g \|_{X'(E)}
\end{equation}
for $f,g\in L^{0}(E)$.

We say that an r.i.~space $X(E)$ is (continuously) \emph{embedded} into an r.i.~space $Y(E)$, and we write $X(E)\to Y(E)$, if $X(E)\subset Y(E)$, and there exists a constant $c$ such that $\|u\|_{Y(E)}\le c\|u\|_{X(E)}$ for every $u\in X(E)$.  Moreover, $X(E)\to Y(E)$ if and only if $Y'(E)\to X'(E)$, and the norms of the two embeddings coincide.

If $|E|<\infty$ and $X(E)$ is any rearrangement-invariant space, then
\begin{equation}\label{E:imm}
    L^\infty (E)\to X(E)\to L^1(E).
\end{equation}

We shall now recall definitions of r.i.~function norms that will be of use later. In what follows, we set $p'=\frac{p}{p-1}$ for $p \in (1,\infty)$, $1'=\infty$ and $\infty'=1$. We also adopt the convention that $\frac 1\infty = 0$.

Pivotal examples of r.i.~function norms are the classical Lebesgue norms
\begin{equation*}
    \|f\|_{{L^p(0,1)}} =
    \begin{cases}
        \|f^\ast\|_{{L^p(0,1)}}&\quad  \text{if} \ p\in[1,\infty)
            \\[0.5ex]
        f^\ast(0^+)&\quad  \text{if} \ p=\infty.
    \end{cases}
\end{equation*}
A generalization of the Lebesgue spaces
is provided by the Orlicz spaces. They are generated by the so-called Luxemburg functionals, whose definition, in turn, rests upon that of Young function. A \emph{Young function}  $A\colon [0, \infty ) \to [0,
\infty ]$ is    a convex,
left-continuous function such that $A(0) =0$ and $A$ is not constant in $(0,  \infty)$. 

\noindent The function $\widetilde{A}\colon [0, \infty ) \to [0,
\infty ]$ denotes the  \emph{Young conjugate}  of $A$, and it is   defined as
\begin{equation}\label{young conj}
    \widetilde{A}(t) = \sup \{\tau t-A(\tau ) \colon \tau \geq 0\}  \quad \ \text{for} \ t\geq 0.
\end{equation} The latter is also a Young function and its conjugate is $A$ again. 
One has that
\begin{equation}\label{young with conj}
    t \leq A^{-1}(t) \, \widetilde A^{-1}(t) \leq 2t \quad \ \text{for} \ t\geq 0,
\end{equation}
where generalized inverses are defined as to be 
right-continuous.

A Young function $A$ is said to \emph{dominate} another Young function $B$ near infinity if positive constants $c$ and $t_0$ exist such that
\begin{equation*}
    B(t)\leq A(c t) \quad \  \text{for} \ t\geq t_0.
\end{equation*}
The functions $A$ and $B$ are called equivalent near infinity if they dominate each other near infinity. The equivalence between the functions $A$ and $B$ will be denoted by
$$ A \simeq B.$$
Given  a~measurable set $E$ of finite positive measure and a Young function $A$, the \emph{Orlicz space} $L^A (E)$ is the rearrangement-invariant space associated with the \emph{Luxemburg function norm} defined  as
\begin{equation*}
    \|f\|_{L^A(0,1)}= \inf \left\{ \lambda >0 \colon  \int_0\sp 1A \left(\frac{f(r)}{\lambda} \right) \, \d r \leq 1 \right\}
\end{equation*}
for $f \in L^0_+(0,1)$. In particular, $L^A (E)= L^p
(E)$ if $A(t)= t^p$ for some $p \in [1, \infty )$, and
$L^A (E)= L^\infty (E)$ if $A(t)= \infty
\chi_{(1, \infty)}(t)$ for $t \in  [0, \infty )$.
One has that
\begin{equation}\label{B.6}
    L^A(E)\to L^B(E) \qquad \text{if and only if} \qquad  A \ \text{dominates} \ B \  \text{near infinity}.
\end{equation} Thus
    $L^A(E)= L^B(E)$ up to equivalent norms  if and only if $A$ is equivalent to $B$ near infinity.
    Also,
    \begin{equation}\label{B.7}
 \left(L^A\right)'(E)   = L^{\widetilde A}(E)
\end{equation} up to equivalent norms.

We denote by $L^p\left(\log L\right)^{\alpha}(E)$ the Orlicz space
associated with a Young function equivalent to $t^p (\log
t)^\alpha$ near infinity, where either $p>1$ and $\alpha \in \mathbb R$, or
$p=1$ and $ \alpha \geq 0$. The space $L^p\left(\log L\right)^{\alpha}(E)$ is also called {\it Zygmund space} in the literature.

\subsection{Spaces of Sobolev type}\label{sec2.3}

Let  $\Omega$ be an
open  set in $\mathbb R^n$, $n \geq 2$, with $|\Omega|<\infty$ and
let $\|\cdot\|_{X(0,1)}$  be  an  r.i.~function norm. For each $m
\in \mathbb N$ , the $m$-th order Sobolev type space $W^mX(\Omega)$
is defined as
\begin{equation}\nonumber
\begin{split}
W^mX(\Omega) = \{ u  \colon   u\    \text{is}
\  \text{$m$-times}  \,  \text{weakly differentiable in}\ 
 \Omega,   \text{and} \   |\nabla^ku| \in X(\Omega) \
  \text{for}  \  k \in \{0, \dots,m\}
\}
\end{split}
\end{equation}
endowed with the norm
$$\| u  \|_{W^m
X(\Omega)} = \sum_{k=0}^m\| \nabla^k u
\|_{X(\Omega)}.$$
Here, 
$|\nabla^k u|$ denotes the Euclidean norm
of $\nabla^k u$ and $\| \nabla^k u
\|_{X(\Omega)}$ is an abridged notation for $\| \,| \nabla^k u
|\, \|_{X(\Omega)}$.

We define the subspace~$W^m_0X(\Omega)$  of~$W^mX(\Omega)$  as
\begin{equation}\nonumber
\begin{aligned}
    W^m_0X(\Omega) = \{ u\in  W^mX(\Omega) \colon   &\text{the continuation of $u$ by $0$ outside $\Omega$} \\&\text{is $m$-times weakly differentiable in $\mathbb R^n$}
\}.
\end{aligned}
\end{equation}
The spaces $W^mX(\Omega)$ and  $W^m_0X(\Omega)$ are Banach spaces. Since $|\Omega|< \infty$,  thanks to a general form of the Poincar\'e inequality for rearrangement-invariant spaces~\cite[Lemma~4.2]{CPbmo}, the space $W^m_0X(\Omega)$ can also be equivalently normed by the functional $ \| \nabla^m u
\|_{X(\Omega)}$ for $u \in W^m_0X(\Omega)$. Thus,
\begin{equation}\label{Equ-Sob-norms}
\| u  \|_{W^m
X(\Omega)} \ \approx \  \| \nabla^m u  \|_{
X(\Omega)}
\end{equation}
for $u \in W^m_0X(\Omega)$, with equivalence constants depending on $X(\Omega)$ and $m$.

Sharp Sobolev embeddings involving r.i.~spaces have a role in our approach. The optimal target r.i.~function  norm in Sobolev type embeddings
for $W^mX(\Omega)$ can be characterized as follows. For each $m \in \mathbb N$, we denote by $\| \cdot \|_{_{X_{m}(0,1)}}$   the rearrangement-invariant function
norm whose associate norm obeys
\begin{equation}\label{E:eucl_opt_norm-john}
    \|f\|_{{X_{m}'(0,1)}}
      =\left\|r\sp{\frac mn}  f\sp{**}(r)\right\|_{{X'(0,1)}}
\end{equation}
for  all $f \in L^0_+ (0,1)$. By~\cite[Theorem~A]{KP1} (see also~\cite[Theorem~6.2]{CPS}), if $\Omega$ is bounded John domain, then
\begin{equation}\label{KP3}
    W^{m}X(\Omega)\  \to \ X_{m}(\Omega),
\end{equation}
and the norm of the embedding depends only on $n, m$ and $\Omega$.
Moreover,  $X_{m}(\Omega)$ is optimal in~\eqref{KP3} among all r.i.~spaces, in the sense that, if~\eqref{KP3} holds with $
X_{m}(\Omega)$ replaced with some other r.i.~space $Y(\Omega)$, then
$$X_{m}(\Omega)\  \to \  Y (\Omega) ,$$ or, equivalently,
\begin{equation}\label{ott}
    X_{m}(0,1)\  \to \  Y (0,1) .
\end{equation}

\begin{lemma}\label{L:1}
    Let $n\in\mathbb N$, $n\ge2$, and $k \in \mathbb N$. Then, there exists a positive constant $C$ depending only on $n$ and $k$ such that, for every 
    $u \in W^{k,1}_{\loc}(\mathbb R^n)$ and every open ball $B\subset\mathbb R^n$, one has
    \begin{equation}\label{E:L1-1}
        \left\|u  - {P}^{k-1}_{B}[u] \right\|_{L^{\frac{n}{n-k}}(B)}
        \leq C\left\|\nabla^{k}u\right\|_{L^1(B)} \quad \  \text{for} \ k\in\{1,\dots, n-1\} 
    \end{equation}
    and
    \begin{equation}\label{E:L1-2}
        \left\|u  - {P}^{k-1}_{B}[u] \right\|_{L^{\infty}(B)}
        \leq C \, |B|^{-1 +\frac{k}{n}}\left\|\nabla^{k}u\right\|_{L^1(B)} \quad \ \text{for} \  k\geq n.
    \end{equation}
\end{lemma}

\begin{proof}
    We may assume, without loss of generality, that balls are centered at $0$. First, let $B$ be the ball,  from this family, such that $|B|= 1$. Then both~\eqref{E:L1-1} and~\eqref{E:L1-2} follow from classical Sobolev embeddings, since $\nabla^{k}u=\nabla^{k}(u-{P}^{k-1}_{B}[u])$ on account of~\eqref{E:moments}.

    Now fix $r\in(0,\infty)$ and consider the ball $B_{r}$, with radius $r$. Given any $u \in W^{k,1}(B_{r})$, we define $u_r=u(r x)$, $x \in B$. Then  $u_r \in W^{k,1}(B)$ and the coefficients of ${P}^{k-1}_{B}[u_r]$ are linear combinations of the components of $\int_B \nabla^h u_r\,\d x$, for $h\in \{0, ..., k-1\}$ with coefficients depending on $n$, $h$ and $B$. Owing to~\eqref{E:moments}, we have ${P}^{k-1}_{B_{r}}[u] (x)  = {P}^{k-1}_{B}[u_r] (r^{-1}x)$ for $x \in \mathbb R^n$. Therefore,  conclusions follow  from the previous case by a scaling argument.
\end{proof}

\section{Technical lemmas}

In this section we collect some basic lemmas concerning one-variable functions, to be used in the proofs of our main results.
 
\begin{lemma}\label{L:standard}
    Let $\theta, \eta \colon(0,\infty)\to(0,\infty)$. Assume that $\theta$ satisfies condition \eqref{E:admissible-new}   with $\varphi$ replaced by $\theta$ and  $\eta$ is non-increasing on $(0,1)$. Fix any $r_0 \in (0,1)$. Then,
\begin{equation}\label{E:standard-assertion}
\sup_{r\in(0,1)}\frac{\eta(r)}{\theta(r)}  \  \approx   \sup_{r\in \left(0,r_0\right)}\frac{\eta(r)}{\theta(r)}
    \end{equation}
   up to constants   depending only on $\theta$ and $r_0$.
\end{lemma}
\begin{proof} The right-hand side of the equivalence \eqref{E:standard-assertion} is trivially bounded by its left-hand side.
\\ To verify the reverse bound,
    denote
    \begin{equation*}
        K = \sup_{r\in\left[r_0,1\right)}\frac{1}{\theta(r)}.
    \end{equation*}
    Since $\theta$ is admissible, we have $K<\infty$. Thus, using also the monotonicity of $\eta$,
    \begin{align*}
        \sup_{r\in(0,1)}\frac{\eta(r)}{\theta(r)}= \max\left\{\sup_{r\in \left(0, r_0\right)}\frac{\eta(r)}{\theta(r)},\sup_{r\in\left[  r_0,1\right)}\frac{\eta(r)}{\theta(r)}\right\}
       \le \max\left\{1, \, K \theta\left(\frac{r_0}{2}\right)\right\}  \sup_{r\in\left(0,r_0\right)}\frac{\eta(r)}{\theta(r)}.
    \end{align*} 
\end{proof}

\begin{lemma}\label{L:double-integral}
  Let  $\alpha\in(0,1)$.  Then,
    \begin{equation}\label{E:double-integral}
       \int_{r}^{1}s^{-1+\alpha}f^{**}(s)\, \d s \ \approx \ r^{-1+\alpha}\int_{0}^{r}f^{*}(s)\, \d s + \int_{r}^{1-r}s^{-1+\alpha}f^{*}(s)\, \d s
    \end{equation}
for every $r\in (0,\tfrac14)$ and $f\in L^0(0,1)$,   with    equivalence constants depending only on $\alpha$.
\end{lemma}

\begin{proof} Fix $r\in (0,\tfrac14)$ and $f\in L^0(0,1)$. 
Then an integration by parts yields
    \begin{equation}\label{L0}
        \int_{r}^{1} s^{-1+\alpha}f^{**}(s)\, \d s =\frac{1}{1-\alpha}\left((r^{-1+\alpha}-1)\int_{0}^{r}f^{*}(s)\, \d s + \int_{r}^{1}(s^{-1+\alpha}-1)f^{*}(s)\, \d s\right).
    \end{equation}
Clearly,
    \begin{equation*}
        \int_{r}^{1} s^{-1+\alpha}f^{**}(s)\, \d s \le \frac{1}{1-\alpha}\left(r^{-1+\alpha}\int_{0}^{r}f^{*}(s)\, \d s + \int_{r}^{1}s^{-1+\alpha}f^{*}(s)\, \d s\right).
    \end{equation*}
   Moreover
    \begin{equation*}
        \int_{r}^{1-r}s^{-1+\alpha}f^{*}(s)\, \d s \ge \int_{1-r}^{1}s^{-1+\alpha}f^{*}(s)\, \d s,
    \end{equation*}
    since the integrand is decreasing and the length of the interval $(1-r,1)$ does not exceed that of $(r,1-r)$.
    Hence
 \begin{equation}\label{L1}
        \int_{r}^{1}s^{-1+\alpha}f^{**}(s)\, \d s \le \frac{2}{1-\alpha}\left(r^{-1+\alpha}\int_{0}^{r}f^{*}(s)\, \d s + \int_{r}^{1-r}s^{-1+\alpha}f^{*}(s)\, \d s\right).
    \end{equation}    
    On the other hand, for $r<\tfrac14$
    \begin{equation*}
    (r^{-1+\alpha}-1)\int_{0}^{r}f^{*}(s)\, \d s \ge (1-4^{-1+\alpha}) r^{-1+\alpha}\int_{0}^{r}f^{*}(s)\, \d s 
    \end{equation*}
       and 
    \begin{equation*}
        \int_{r}^{1-r}(s^{-1+\alpha}-1)f^{*}(s)\, \d s \ge \int_{r}^{\frac 12}(s^{-1+\alpha}-1)f^{*}(s)\, \d s \ge (1-2^{-1+\alpha})\int_{r}^{\frac 12}s^{-1+\alpha}f^{*}(s)\, \d s,
    \end{equation*}  since $1-r>\frac12$. 
    Moreover
        \begin{equation*}
        \int_{r}^{1-r} s^{-1+\alpha} f^{*}(s)\, \d s \le 2\int_{r}^{\frac 12}s^{-1+\alpha}f^{*}(s)\, \d s, 
    \end{equation*} 
    since the intervals $(r,\tfrac 12)$ and $(\tfrac12,1-r)$ are of equal length, and the integrand is decreasing.\\
        Consequently, by~\eqref{L0},
   \begin{equation}\label{L2}
        \int_{r}^{1}s^{-1+\alpha}f^{**}(s)\, \d s \ge \frac{1-2^{-1+\alpha}}{2(1-\alpha)}
        \left(r^{-1+\alpha}\int_{0}^{r}f^{*}(s)\, \d s + \int_{r}^{1-r}s^{-1+\alpha}f^{*}(s)\, \d s\right).
    \end{equation}
  Coupling  inequalities~\eqref{L1} and~\eqref{L2}  provides~\eqref{E:double-integral}.
\end{proof}

\begin{lemma}\label{L:interval-norm}
  Let  $\|\cdot\|_{X(0,1)}$ be a rearrangement-invariant function norm and 
let  $\alpha\in(0,\infty)$. Then,
    \begin{equation}\label{E:interval-norm}
       \|s^{\alpha}\chi_{(0,r)}^{**}(s)\|_{X(0,1)} \ \approx \ r  \, \|s^{-1+\alpha}\chi_{(r,1)}(s)\|_{X(0,1)}
    \end{equation}
    for every $r\in  (0,\tfrac{1}{2})$, up to a constant depending only on $\alpha$.
\end{lemma}

\begin{proof} Fix $r\in  (0,\tfrac{1}{2})$. Then,
    \begin{equation*}
      s^{\alpha}\chi_{(0,r)}^{**}(s) = s^{\alpha}\chi_{(0,r)}(s)+rs^{-1+\alpha}\chi_{(r,1)}(s) \quad \ \text{for} \ s\in(0,1).
    \end{equation*}
   The bound of the right-hand side of~\eqref{E:interval-norm} by its left-hand side   is a straightforward consequence of ~\eqref{N2}. To verify the reverse bound,
     note that
    \begin{align*}
        \|s^{\alpha}\chi_{(0,r)}^{**}(s)\|_{X(0,1)} & \leq  \|s^{\alpha}\chi_{(0,r)}(s)\|_{X(0,1)}+r  \, \|s^{-1 +\alpha}\chi_{(r,1)}(s)\|_{X(0,1)}.
    \end{align*}
   Thus,
    \begin{align*}
    \|s^{\alpha}\chi_{(0,r)}(s)\|_{X(0,1)} & \leq  r^{\alpha}  \, \|\chi_{(0,r)}\|_{X(0,1)} =
    r^{\alpha}  \, \|\chi_{(r,2r)}\|_{X(0,1)}
    \leq \max\{2^{1-\alpha},1\}  \,  r  \, \| s^{-1+\alpha}\chi_{(r,2r)}(s)\|_{X(0,1)}
        \\
    &\leq  \max\{2^{1-\alpha},1\} \,   r \, \|s^{-1 +\alpha}\chi_{(r,1)}(s)\|_{X(0,1)}.
    \end{align*}
    \end{proof}

\begin{lemma}\label{L:interpolation-lemma}
    Let $\gamma\in (-\infty,1)$ and let $\|\cdot\|_{X(0,1)}$ be a rearrangement-invariant function norm. Then, the operator $T_{\gamma}$, defined as
    \begin{equation*}\label{E:operator-t-gamma}
        T_{\gamma }f(r) = r^{1-\gamma}\int_{r}^{1} s^{-2+\gamma}  f(s) \,\d s \quad \ \text{for} \  r\in(0,1) \ \mbox{and} \ f\in L^0_+(0,1), 
    \end{equation*}
    is bounded on $X(0,1)$, and
    \begin{equation}\label{E:Cald Tg}
        \|T_{\gamma}\|_{X(0,1)\to X(0,1)} \ \leq \ \tfrac{1}{1-\gamma}.
    \end{equation}
\end{lemma}
\begin{proof}
    By~\cite[Chapter~3, Theorem~2.2]{BS}, $X(0,1)$ is an exact interpolation space between $L^1(0,1)$ and $L^{\infty}(0,1)$. Moreover, using the Fubini's theorem and elementary integration, we get
    \begin{equation*}
        \|T_{\gamma}f\|_{L^1(0,1)}= \tfrac{1}{2-\gamma} \, \|f\|_{L^1(0,1)}\quad \ \text{for} \ f\in L^0_+(0,1),
    \end{equation*}
    and, owing to a trivial estimate,
    \begin{equation*}
        \|T_{\gamma}f\|_{L^{\infty}(0,1)} \le \tfrac{1}{1-\gamma} \, \|f\|_{L^{\infty}(0,1)}\quad \  \text{for} \ f\in L^0_+(0,1).
    \end{equation*}
     Thanks to an interpolation theorem by Calder\'on \cite[Chapter 3, Theorem 2.12]{BS}, $T_{\gamma}$ is bounded on $X(0,1)$, and~\eqref{E:Cald Tg} holds.
\end{proof}
 
\begin{lemma}\label{L:half-interval}
 Let $\|\cdot\|_{X(0,1)}$ be a rearrangement-invariant function norm and   let  $\alpha\in\mathbb R$. Then,
    \begin{equation}\label{E:half-interval}
       \|s^{\alpha}\chi_{(r,1)}(s)\|_{X(0,1)} \ \approx \ \|s^{\alpha}\chi_{\left(r,\frac 12\right)}(s)\|_{X(0,1)}
    \end{equation}
    for every $r\in  (0,\tfrac{1}{4})$, up a constant depending only on $\alpha$.  
\end{lemma}

\begin{proof} The bound of the right-hand side of~\eqref{E:half-interval} via its left-hand side is trivial. In order to prove the opposite bound, fix $r\in\left(0,\frac{1}{4}\right)$. Then,
    \begin{equation*}
        \|s^{\alpha}\chi_{(r,1)}(s)\|_{X(0,1)} \le \|s^{\alpha}\chi_{\left(r,\frac12\right)}(s)\|_{X(0,1)} + \|s^{\alpha}\chi_{\left(\frac12,1\right)}(s)\|_{X(0,1)}.
    \end{equation*}
    Next, since $2r\le\frac12$, one has
    \begin{align*}
        \|s^{\alpha}\chi_{\left(\frac12,1\right)}(s)\|_{X(0,1)} & \le  \|s^{\alpha}\chi_{(2r,1)}(s)\|_{X(0,1)} = \|s^{\alpha}\chi_{\left(r,\frac12\right)}\left(\tfrac{s}{2}\right)\|_{X(0,1)}
            \\
        &= 2^{\alpha} \|\left(\tfrac{s}{2}\right)^{\alpha}\chi_{\left(r,\frac 12\right)}\left(\tfrac{s}{2}\right)\|_{X(0,1)} \le 2^{\alpha}\|s^{\alpha}\chi_{\left(r,\frac 12\right)}(s)\|_{X(0,1)},
    \end{align*}
    where the last inequality holds because of the boundedness of the dilation operator~$\mathcal E_2$ on $X(0,1)$, with norm not exceeding $1$. Hence, 
the bound of the left-hand side of~\eqref{E:half-interval} via  its right-hand times $2^\alpha +1$ follows.
\end{proof}

\begin{lemma}\label{L:comparison-of-conditions}
   Let $\|\cdot\|_{X(0,1)}$ be a rearrangement-invariant function norm and   let  $\alpha\in  [0,\infty)$. Then,
    \begin{equation}\label{E:comparison-of-conditions}
        \|s^{\alpha}\chi_{(0,r)}(s)\|_{X(0,1)} \ \approx  \ r^{\alpha}\|\chi_{(0,r)}\|_{X(0,1)}
    \end{equation}
    for every $r\in (0,1)$, with equivalence constants depending only on $\alpha$.
\end{lemma}

\begin{proof} The assertion follows at once from the chain
    \begin{align*}
        \left\|s^{\alpha}\chi_{(0,r)}(s)\right\|_{{X(0,1)}} &\le r^{\alpha}\left\|\chi_{(0,r)}\right\|_{{X(0,1)}}
        \le 2 r^{\alpha}\left\|\chi_{\left(\frac{r}{2},r\right)}\right\|_{{X(0,1)}}
            \\
        &\lesssim \left\|s^{\alpha}\chi_{\left(\frac{r}{2},r\right)} (s)\right\|_{X(0,1)}
        \le \left\|s^{\alpha}\chi_{(0,r)}(s)\right\|_{X(0,1)}
    \end{align*}   
    for every $r\in (0,1)$, in which the constant in $\lesssim$ depends only on $\alpha$.
\end{proof}

\section{Proofs of embeddings into Morrey spaces}

The present section is devoted to the proof of the main results from Subsection~\ref{S:morrey}.   

\begin{proof}[Proof of Theorem~\ref{T:2}]
The full statement of this theorem follows from the circle of implications  $(i) \Rightarrow (iii) \Rightarrow (ii) \Rightarrow (i)$, which are established below. However, we also present a direct proof of the implication $(iii) \Rightarrow (i)$. This provides us with a self-contained proof of the equivalence  $(i) \Leftrightarrow (iii)$, which avoids calling into play Marcinkiewicz spaces and, in particular, the characterization of the optimal target norm \eqref{E:eucl_opt_norm-john} in Sobolev embeddings into rearrangement-invariant spaces.
\\ Unless otherwise stated, throughout this proof the constants in the relations \lq\lq $\lesssim$'' and \lq\lq $\gtrsim$'' only depend on $n$, $m$ and $\varphi$.
\\ 
$(iii) \Rightarrow (i)$ As recalled in Remark~\ref{rem-m=1}, when $m=1$ this implication is established in~\cite[Theorem~1.2]{CPcamp}. In that paper, the function  $\varphi$ is assumed to be continuous and balls are replaced with cubes in the definition of the Morrey norm, but the same proof applies in the current framework. We may thus focus on the case when $m\ge2$.  By~\cite[Theorem~6.2]{CPS}, one has that
    \begin{equation}\label{E:mINT-emb}
        W^{m}X (\Omega) \to W^{1}X_{m-1} (\Omega),
    \end{equation}
    where $\|\cdot\|_{X_{m-1}(0,1)}$ is the r.i.~function norm defined as in \eqref{E:eucl_opt_norm-john}, with $m$ replaced with $m-1$.
 
 \noindent   Consequently, the conclusion will follow if we show that
    \begin{equation}\label{E:first-order-morrey}
        W^{1}X_{m-1}(\Omega)\to \mathcal M^{\varphi(\cdot)}(\Omega).
    \end{equation}
    By the result for $m=1$, the embedding~\eqref{E:first-order-morrey} holds if (and only if)
    \begin{equation}\label{E:th2ACPmor}
        \sup_{r \in (0,1)}\frac{1}{\varphi(r^{\frac1n})}\left\|s^{-1+\frac1n}\chi_{(r,1)}(s)\right\|_{X'_{m-1}(0,1)}<\infty.
    \end{equation}
    Owing to~\eqref{E:eucl_opt_norm-john}, the latter condition can be rewritten as
    \begin{equation}\label{E:th2ACPmor2}
        \sup_{r \in (0,1)}\frac{1}{\varphi(r^{\frac1n})}\left\|s^{\frac{m-1}{n}}\left((\cdot)^{-1+\frac1n}\chi_{(r,1)}(\cdot)\right)^{**}(s)\right\|_{X'(0,1)}<\infty.
    \end{equation}
    Fix $r\in\left(0,\frac12\right)$ and define
    \begin{equation*}
        f(s)= s^{-1+\frac1n}\chi_{(r,1)}(s) \quad \ \text{for} \ s\in (0,1).
    \end{equation*}
    Then
    \begin{equation*}
        f^*(\varrho)=(\varrho+r)^{-1+\frac1n}\chi_{(0,1-r)}(\varrho) \quad \ \text{for} \ \varrho\in (0,1),
    \end{equation*}
    and, since $r<1-r$,
    \begin{align*}
        f^{**}(s) &= \frac{1}{s}\int_{0}^{s}(\varrho+r)^{-1+\frac1n}\chi_{(0,1-r)}(\varrho)\,\d \varrho
            \\
        &\le \frac{\chi_{(0,r)}(s)}{s}\int_{0}^{s}r^{-1+\frac1n}\,\d \varrho + \frac{\chi_{(r,1-r)}(s)}{s}\int_{0}^{s}\varrho^{-1+\frac1n}\,\d \varrho
           + \frac{\chi_{(1-r,1)}(s)}{s}\int_{0}^{1-r}(\varrho+r)^{-1+\frac1n}\,\d \varrho
            \\
        &\le r^{-1+\frac1n} \chi_{(0,r)}(s) + ns^{-1+\frac1n} \chi_{(r,1-r)}(s) + ns^{-1}\chi_{(1-r,1)}(s)\quad \ \text{for} \ s\in (0,1).
    \end{align*}
    Therefore,
    \begin{align*}
        &\left\|s^{\frac{m-1}{n}} \left((\cdot)^{-1+\frac1n}\chi_{(r,1)}(\cdot)\right)^{**}(s)\right\|_{X'(0,1)}
            \\
        &\qquad \qquad \le \left\|s^{\frac{m-1}{n}}\left(r^{-1+\frac1n}\chi_{(0,r)}(s) + ns^{-1+\frac1n}
        \chi_{(r,1-r)}(s)+ ns^{-1}\chi_{(1-r,1)}(s)\right)\right\|_{X'(0,1)}
           \\
        &\qquad \qquad \le r^{-1+\frac1n}\left\|s^{\frac{m-1}{n}}\chi_{(0,r)}(s)\right\|_{X'(0,1)} + n \left\|s^{-1 +\frac{m}{n}}\chi_{(r,1-r)}(s)\right\|_{X'(0,1)}
        + n \left\|s^{-1 +\frac{m-1}{n}}\chi_{(1-r,1)}(s)\right\|_{X'(0,1)}.
    \end{align*}
    We claim that each addend  on the rightmost side of this inequality is bounded above, up to multiplicative constants, by $ \left\|s^{-1 +\frac{m}{n}}\chi_{(r,1)}(s)\right\|_{X'(0,1)}$. Indeed, while this is clearly true for the  middle one, one has 
    \begin{align*}
        \left\|s^{\frac{m-1}{n}}\chi_{(0,r)}(s)\right\|_{X'(0,1)} &\le r^{\frac{m-1}{n}}\left\|\chi_{(0,r)}\right\|_{X'(0,1)} = r^{\frac{m-1}{n}}\left\|\chi_{(r,2r)}\right\|_{X'(0,1)}\\
        &        \le r^{\frac{m-1}{n}}(2r)^{1-\frac{m}{n}}\left\|s^{-1+\frac{m}{n}}\chi_{(r,2r)}(s)\right\|_{X'(0,1)}
             = 2^{1-\frac{m}{n}} r^{1-\frac{1}{n}}\left\|s^{-1 +\frac{m}{n}}\chi_{(r,2r)}(s)\right\|_{X'(0,1)}\\
        &
        \le 2^{1-\frac{m}{n}} r^{1-\frac{1}{n}}\left\|s^{-1 +\frac{m}{n}}\chi_{(r,1)}(s)\right\|_{X'(0,1)},
    \end{align*}
   and
    \begin{align*}
        \left\|s^{-1 +\frac{m-1}{n}}\chi_{(1-r,1)}(s)\right\|_{X'(0,1)}
        &= \left\|s^{-\frac{1}{n}} s^{-1 +\frac{m}{n}}  \chi_{(1-r,1)}(s)\right\|_{X'(0,1)}\le (1-r)^{-\frac{1}{n}}\left\|s^{-1 +\frac{m}{n}}\chi_{(1-r,1)}(s)\right\|_{X'(0,1)}
            \\
        &\le 2^{\frac{1}{n}}\left\|s^{-1 +\frac{m}{n}}\chi_{(r,1)}(s)\right\|_{X'(0,1)}.
    \end{align*}
    Altogether,
    \begin{equation*}
       \left\|s^{\frac{m-1}{n}}\left((\cdot)^{-1+\frac1n}\chi_{(r,1)}(\cdot)\right)^{**}(s)\right\|_{X'(0,1)}
       \lesssim \left\|s^{-1 +\frac{m}{n}}\chi_{(r,1)}(s)\right\|_{X'(0,1)}.
    \end{equation*}
    Consequently,
    \begin{align*}
        \sup_{r \in \left(0,\frac12\right)} \frac{1}{\varphi(r^{\frac1n})}\left\|s^{\frac{m-1}{n}}\left((\cdot)^{-1+\frac1n}\chi_{(r,1)}(\cdot)\right)^{**}(s)\right\|_{X'(0,1)}
        &\lesssim \sup_{r \in \left(0,\frac12\right)} \frac{1}{\varphi(r^{\frac1n})}\left\|s^{-1 +\frac{m}{n}}\chi_{(r,1)}(s)\right\|_{X'(0,1)}
            \\
         &\le \sup_{r \in \left(0,1\right)} \frac{1}{\varphi(r^{\frac1n})}\left\|s^{-1 +\frac{m}{n}}\chi_{(r,1)}(s)\right\|_{X'(0,1)}.
    \end{align*}  
    Hence, by the very definition of the function norm $\left\|\,\cdot\,\right\|_{X'_{m-1}(0,1)}$,
    \begin{align*}
        \sup_{r \in \left(0,\frac12\right)}\frac{1}{\varphi(r^{\frac1n})}\left\|s^{-1+\frac1n}\chi_{(r,1)}(s)\right\|_{X'_{m-1}(0,1)}
       \   \lesssim \ \sup_{r\in\left(0,1\right)}\frac{1}{\varphi(r^{\frac1n})}\left\|s^{-1 +\frac{m}{n}}\chi_{(r,1)}(s)\right\|_{X'(0,1)}.
    \end{align*}  
    Applying Lemma~\ref{L:standard}  with $\theta(r)={\varphi(r^{\frac1n})}$  and $\eta(r)=\left\|s^{-1+\frac{1}{n}}\chi_{\left(r,1\right)}(s)\right\|_{X'_{m-1}(0,1)}$, and making use of the assumption~\eqref{E:admissible-new}, we ensure that
    \begin{align*}
        \sup_{r \in \left(0,1\right)}\frac{1}{\varphi(r^{\frac1n})}\left\|s^{-1+\frac1n}\chi_{(r,1)}(s)\right\|_{X'_{m-1}(0,1)}
        \  \lesssim \ \sup_{r\in\left(0,1\right)}\frac{1}{\varphi(r^{\frac1n})}\left\|s^{-1 +\frac{m}{n}}\chi_{(r,1)}(s)\right\|_{X'(0,1)}.
    \end{align*}
    Since the right-hand side is finite by~\eqref{E:7}, this establishes~\eqref{E:th2ACPmor}, and the latter equation in turn implies~\eqref{E:first-order-morrey}.
    \\
    $(i) \Rightarrow (iii)$.  For simplicity of notation, we shall prove   this  implication under the assumption that  $\Omega$ is the ball   centered at $0$  with $|\Omega|=1$. Thanks to Remark~\ref{rem-nov6}, it will be clear that an analogous argument applies by replacing this ball with a sufficiently small ball contained in any domain $\Omega$.
\\ Given a nonnegative function~$f\in X(0,1)$, define the function $v_f\colon \Omega\to\mathbb R$ as
    \begin{equation}\label{E:new5}
        v_f(x)=  \int_{\omega_n |x|\sp n}\sp{1}r\sp{-m+\frac mn} f(r)(r-\omega_n|x|\sp n)^{m-1}\,\d r \quad \ 
 \text{for} \ x\in \Omega,
            \end{equation}
            where $\omega_n$ denotes the Lebesgue measure of the unit ball.
    Then, $v_f$ is   a nonnegative  radially decreasing and $m$-times weakly-differentiable  function. Similarly as in~\cite[proof of Theorem~A]{KP1}, it can be shown that $v_f\in W^mX(\Omega)$ and
    \begin{equation}\label{E:Tmo2}        \left\|v_f\right\|_{W^mX(\Omega)} \ \lesssim \ \|f\|_{_{X(0,1)}}.
    \end{equation}
  We claim that
    \begin{equation}\label{E:T3}
        \dashint_{B(r)}v_f \,\d x \ \gtrsim \  \int_{r}^{1} s\sp{-1+\frac{m}{n}} f(s)\,\d s \quad \ \text{for} \ r\in \left(0,1\right),
    \end{equation}   
    where  
    \begin{align}
        \label{fi5}
     B(r)  \   \text{is the ball, centered at} \ 0,  \   \text{with measure} \ r.
    \end{align}
    Indeed, given $r\in \left(0,1\right)$, changing variables and using the Fubini theorem yields the following chain:
    \begin{align*}
        \dashint_{B(r)} v_f \,\d x &=\dashint_{B(r)}\int_{\omega_n |x|\sp n}\sp{1} s\sp{-m+\frac mn} f(s)(s-\omega_n|x|\sp n)^{m-1}\,\d s\,\d x
            \\
        &=   \frac{n \omega_n}{r}\int_{0}^{(\omega_n^{-1}r)^{\frac1n}}\rho^{n-1}\int_{\omega_n \rho\sp n}\sp{1}s\sp{-m+\frac mn} f(s)(s-\omega_n \rho\sp n)^{m-1}\,\d s\,\d \rho
            \\
        &=   \frac{1}{r} \int_{0}^{r}\int_{\varrho}\sp{1} s\sp{-m+\frac mn} f(s)(s-\varrho)^{m-1}\,\d s\,\d\varrho
            \\
        &\ge  \frac{1}{r}  \int_{0}^{r}\int_{r}\sp{1}s\sp{-m+\frac mn} f(s)(s-\varrho)^{m-1}\,\d s\,\d\varrho
            \\
        &=  \frac{1}{r} \int_{r}\sp{1} s\sp{-m+\frac mn} f(s)\int_{0}^{r}(s-\varrho)^{m-1}\,\d\varrho\,\d s
            \\
        &\ge  \frac{1}{r}  \int_{r}\sp{1}s\sp{-m+\frac mn} f(s)\int_{0}^{\frac{r}{2}}(s-\varrho)^{m-1}\,\d\varrho\,\d s
            \\
        &\ge \frac{1}{2^m} \int_{r}\sp{1}s\sp{-1+\frac mn} f(s)\,\d s.
    \end{align*}
    The inequality~\eqref{E:T3} is thus established.
    \\ From Lemma~\ref{L:standard} applied with $\theta(r)={\varphi(r^{\frac1n})}$ and $\eta(r)=\left\|s^{-1+\frac{m}{n}}\chi_{\left(r,1\right)}(s)\right\|_{X'(0,1)}$, and the assumption~\eqref{E:admissible-new} we obtain that
    \begin{equation*}
        \sup_{r\in \left(0,1\right)}\frac{1}{\varphi(r^{\frac1n})}\left\|s^{-1+\frac{m}{n}}\chi_{\left(r,1\right)}(s)\right\|_{X'(0,1)} \ 
       \lesssim \ \sup_{r\in \left(0,\frac{1}{2}\right)}\frac{1}{\varphi(r^{\frac1n})}\left\|s^{-1+\frac{m}{n}}\chi_{\left(r,1\right)}(s)\right\|_{X'(0,1)}.
    \end{equation*}  
    Hence, via~\eqref{E:Tmo2},~\eqref{E:T3}, and~\eqref{E:SO-MAemb}, one infers that
    \begin{align*}
        \sup_{r\in \left(0,1\right)}\frac{1}{\varphi(r^{\frac1n})}\left\|s^{-1+\frac{m}{n}}\chi_{\left(r,1\right)}(s)\right\|_{X'(0,1)}
        &\lesssim \sup_{r\in\left(0,\frac{1}{2}\right)}\frac{1}{\varphi(r^{\frac1n})}\sup_{f \in X_+(0,1)}\frac{\int_{r}\sp{1}  s\sp{-1+\frac mn} f(s)\,\d s}{\|f\|_{X(0,1)}}
            \\
        &\lesssim \sup_{f \in X_+(0,1)}\frac{1}{\|f\|_{X(0,1)}}\sup_{r\in\left(0,\frac{1}{2}\right)}\frac{1}{\varphi(r^{\frac1n})}\dashint_{B(r)} v_f \,\d x
            \\
        &\le \sup_{f \in X_+(0,1)}\frac{\|v_f \|_{\mathcal M^{\varphi(\cdot)}(\Omega)}}{\|f\|_{X(0,1)}}
            \\
        &\lesssim \sup_{f\in X_+(0,1)}\frac{\|v_f\|_{\mathcal M^{\varphi(\cdot)}(\Omega)}}{\|v_f\|_{W^mX(\Omega)}}
                 \\
        &\le \sup_{u\neq 0}\frac{\|u\|_{\mathcal M^{\varphi(\cdot)}(\Omega)}}{\|u\|_{W^mX(\Omega)}} < \infty.
    \end{align*} Hence,~\eqref{E:7} follows.
\\ $(ii) \Rightarrow (i)$. This is a trivial consequence of   \eqref{marc-mor}.
    \\ $(iii) \Rightarrow (ii)$. Assume that~\eqref{E:7} holds and define the function~$\psi$ by
    \begin{equation*}
        \psi(r)=\|s^{-1+\frac{m}{n}}\chi_{(r^n,1)}(s)\|_{X'(0,1)}\quad \ \text{for} \ r\in(0,1).
    \end{equation*}
    By~\cite[Theorem~6.2]{CPS} (in case when $\Omega$ is Lipschitz, see also~\cite[Theorem~A]{KP1}),
    \begin{equation}\label{E:mINT-embm}
        W^{m}X (\Omega) \to X_{m} (\Omega),
    \end{equation}
    where $\|\cdot\|_{X_{m}(0,1)}$ is the r.i.~function norm defined by \eqref{E:eucl_opt_norm-john}.
   Hence, in particular,  $\|\chi_{(0,r)}\|_{X'_{m}(0,1)}=\|s^{\frac{m}{n}} \chi_{(0,r)}^{**}(s)\|_{X'(0,1)}$ for every $r\in(0,1)$. By Lemma~\ref{L:interval-norm},  
    \begin{equation*}
        \|\chi_{(0,r)}\|_{X'_{m}(0,1)} \ \approx \  r \, \|s^{-1 +\frac{m}{n}} \chi_{(r,1)}(s)\|_{X'(0,1)}\quad \ \text{for} \ r\in (0,\tfrac12).
    \end{equation*}
    Therefore, thanks to equation~\eqref{feb100},
    $$\|\chi_{(0,r)}\|_{X_{m}(0,1)} \ \approx \ \frac{1}{\psi(r^{\frac{1}{n}})} \quad \ \text{for $r\in\left(0,\tfrac12\right)$.}$$  Since 
    $$\|\chi_{(0,r)}\|_{\mathfrak M^{\psi(\cdot)}(0,1)} = \frac{1}{\psi(r^{\frac{1}{n}})},$$
    and 
    the Marcinkiewicz space $\mathfrak M^{\psi(\cdot)}(\Omega)$ is the largest space 
    among all r.i.~spaces 
     which share the same fundamental function,
     we infer 
      that $W^mX(\Omega)\to \mathfrak M^{\psi(\cdot)}(\Omega)$. Since~\eqref{E:7} implies that $\psi\lesssim\varphi$, the embedding~\eqref{E:SO-MARCemb} follows.
\end{proof}

\begin{proof}[Proof of Theorem~\ref{T:morrey-optimal-range}]
    The very definition of the function~$\widehat\varphi$ ensures that it is admissible and
    \begin{equation*}
        \sup_{r\in\left(0,1\right)}\frac{1}{\widehat\varphi(r)}\left\|s^{-1 +\frac{m}{n}}\chi_{(r^n,1)}(s)\right\|_{X'(0,1)}=1.
    \end{equation*}
    Thus, the embedding~\eqref{E:optimal-morrey-range-embedding} follows via Theorem~\ref{T:2}. 
    \\ To prove the optimality of $\mathcal M^{\widehat\varphi(\cdot)}(\Omega)$ among Morrey spaces, suppose that $\varphi\colon(0,\infty)\to(0,\infty)$ is an admissible function satisfying $W^m X(\Omega) \to \mathcal M^{\varphi(\cdot)}(\Omega)$. Then, by Theorem~\ref{T:2}, the condition~\eqref{E:7} must be satisfied. Hence, the definition of~$\widehat\varphi$ tells us that
        \begin{equation*}
        \widehat\varphi(r) \ \lesssim \ \varphi(r)\quad \ \text{for} \ r\in (0,\tfrac12)
    \end{equation*}
    up to a constant independent of $r$, whence the embedding~$\mathcal M^{\widehat\varphi(\cdot)}(\Omega) \to \mathcal M^{\varphi(\cdot)}(\Omega)$ follows.
\end{proof}

\begin{proof}[Proof of Theorem~\ref{T:morrey-optimal-domain}] Unless otherwise stated, throughout this proof  the relations \lq\lq $\lesssim$'' and \lq\lq $ \gtrsim$'' hold up to constants depending on $n$, $m$ and $\varphi$.
    The axioms of the definition of rearrangement-invariant function norm for 
    $\|\cdot\|_{X^{\#}(0,1)}$ are easily verified.
    Let us just notice that  property~\eqref{N4} is a consequence of the fact that
    \begin{equation*}
        \|\chi_{(0,1)}\|_{X^{\#}(0,1)} \ \approx \  \sup_{r\in(0,1)}\frac{1-r^{\frac{m}{n}}}{\varphi(r^{\frac{1}{n}})}\  \le \ \sup_{r\in(0,1)}\frac{1}{\varphi(r^{\frac{1}{n}})},
    \end{equation*}
   since the latter supremum is finite by~\eqref{E:vanishing-condition}. 
   \\ Now, assume that $\|\cdot\|_{X(0,1)}$ is a rearrangement-invariant function norm. We claim that
    \begin{equation}\label{E:optimality-morrey-chain}
        \sup_{r\in(0,1)}\frac{1}{\varphi(r^{\frac{1}{n}})}  \|s^{-1+\frac{m}{n}}\chi_{(r,1)}(s)\|_{X'(0,1)}
        \ \approx \  \sup_{\|f\|_{X(0,1)}\le1}\|f\|_{X^{\#}(0,1)}.
    \end{equation}
    To verify this claim,
   fix $r\in (0,\tfrac14)$ and observe that, by~\eqref{n.assoc.2} and Lemma~\ref{L:double-integral}, 
    \begin{align*}
        \|s^{-1+\frac{m}{n}}\chi_{(r,1)}(s)\|_{X'(0,1)} &= \sup_{\|f\|_{X(0,1)}\le1} \int_{0}^{1}\left[(\cdot)^{-1+\frac{m}{n}}\chi_{(r,1)}(\cdot)\right]^*(s) f^{*}(s) \, \d s
            \\
        &= \sup_{\|f\|_{X(0,1)}\le1} \int_{0}^{1}(s+r)^{-1+\frac{m}{n}} \chi_{(0,1-r)}(s) f^{*}(s) \, \d s
            \\
        &\approx \sup_{\|f\|_{X(0,1)}\le1}\left(r^{-1+\frac{m}{n}}\int_{0}^{r}f^{*}(s)\,\d s+\int_{r}^{1-r}s^{-1+\frac{m}{n}}f^{*}(s)\,\d s\right)
            \\
        & \approx \sup_{\|f\|_{X(0,1)}\le1}\int_{r}^{1}s^{-1+\frac{m}{n}}f^{**}(s) \, \d s.
    \end{align*}  
    Consequently,
    \begin{equation*}
        \sup_{r\in \left(0,\frac14\right)} \,  \frac{1}{\varphi(r^{\frac{1}{n}})}\|s^{-1+\frac{m}{n}}\chi_{(r,1)}(s)\|_{X'(0,1)}
        \ \approx \ \sup_{\|f\|_{X(0,1)}\le1}\sup_{r\in \left(0,\frac14\right)} \, \frac{1}{\varphi(r^{\frac{1}{n}})} \int_{r}^{1}s^{-1+\frac{m}{n}}f^{**}(s) \, \d s.
    \end{equation*}
    Lemma~\ref{L:standard} ensures  that
    \begin{equation*}
        \sup_{r\in(0,1)} \,  \frac{1}{\varphi(r^{\frac{1}{n}})}\|s^{-1+\frac{m}{n}}\chi_{(r,1)}(s)\|_{X'(0,1)}
        \ \approx \ \sup_{\|f\|_{X(0,1)}\le1}\sup_{r\in(0,1)}\,  \frac{1}{\varphi(r^{\frac{1}{n}})}\int_{r}^{1}s^{-1+\frac{m}{n}}f^{**}(s)\d s.
    \end{equation*}
    Hence, equation~\eqref{E:optimality-morrey-chain} follows.
    If $X(0,1)=X^{\#}(0,1)$, then
    \begin{equation*}
\sup_{\|f\|_{X(0,1)}\le1}\|f\|_{X^{\#}(0,1)}=1,
    \end{equation*}
    whence, owing to~\eqref{E:optimality-morrey-chain}, the condition~\eqref{E:7} is satisfied. Therefore, by Theorem~\ref{T:2}, embedding~\eqref{E:optimal-morrey-domain-embedding} holds.\\
    We next focus on  the optimality of $X^{\#}(\Omega)$ in~\eqref{E:optimal-morrey-domain-embedding}. If $W^m X (\Omega) \to \mathcal M^{\varphi(\cdot)} (\Omega)$  for some  re\-ar\-ran\-ge\-ment-in\-va\-ri\-ant function norm $\|\cdot\|_{X(0,1)}$, then~\eqref{E:7} holds owing to  Theorem~\ref{T:2}. Thus, by~\eqref{E:optimality-morrey-chain}, 
    \begin{equation*}    \sup_{\|f\|_{X(0,1)}\le1}\|f\|_{X^{\#}(0,1)}<\infty,
    \end{equation*}
    a piece of information which immediately implies the embedding $X(0,1) \to X^{\#}(0,1)$.\\
    Finally, if~\eqref{E:vanishing-condition} is not satisfied, then with the choice $\|\cdot\|_{X(0,1)}=\|\cdot\|_{L^\infty(0,1)}$, one has that $\|\cdot\|_{X'(0,1)}=\|\cdot\|_{L^1(0,1)}$, whence
    \begin{equation*}
        \sup_{r \in (0,1)} \,  \frac{1}{\varphi( r^{\frac1n})}\left\| s^{-1+\frac{m}{n}}\chi_{{(r,1)}}(s)\right\|_{{X'(0,1)}}
        \ \approx \  \sup_{r\in(0,1)}\,  \frac{1-r^{\frac{m}{n}}}{\varphi(r^{\frac{1}{n}})} = \infty,
    \end{equation*}
    and condition~\eqref{E:7} fails. Thus, by Theorem~\ref{T:2}, the embedding
    \begin{equation*}
        W^mL^{\infty}(\Omega) \to \mathcal M^{\varphi(\cdot)}(\Omega)
    \end{equation*}
    cannot hold. Owing to~\eqref{E:imm}, there does not exist any rearrangement-invariant function norm $\|\cdot\|_{X(0,1)}$ that renders the embedding
    \begin{equation*}
        W^mX(\Omega) \to \mathcal M^{\varphi(\cdot)}(\Omega)
    \end{equation*}
    true.
\end{proof}

\begin{proof}[Proof of Theorem~\ref{T:sobolev-morrey-vanishing}]
 \eqref{E:sobolev-to-morrey-vanishing-condition} $\Rightarrow$ \eqref{fi2}. Let $\widehat\varphi$ be the  function defined by~\eqref{E:optimal-morrey-range}. Owing to~\eqref{E:optimal-morrey-range-embedding},
    \begin{equation}\label{PA1}
        \sup_{B\subset\Omega} \, \frac{1}{\widehat\varphi(|B|^{\frac{1}{n}})}\dashint_B|u|\, \d x \ \lesssim \ \|u\|_{W^mX(\Omega)},
    \end{equation}
   up to a constant independent of $u\in W^mX(\Omega)$. Consequently, for every function $u\in W^mX(\Omega)$ such that $\|u\|_{W^mX(\Omega)}\le1$ and every ball $B\subset\Omega$, we have
   \begin{equation}\label{PA2}
        \dashint_B|u|\, \d x \  \lesssim \ \widehat\varphi(|B|^{\frac{1}{n}}),
    \end{equation}
     up to a constant independent of $u\in W^mX(\Omega)$ and of $B$.
\\
    Fix $r\in(0,1)$. Then, via~\eqref{PA1} and~\eqref{PA2},
    \begin{equation}\label{E:thm1.4-balls}
        \begin{split}
            \sup\left\{\psi_{\varphi,u}(r)\colon \|u\|_{W^mX(\Omega)}\le1\right\}
            &= \sup\left\{\sup_{B\subset\Omega,|B|\le r} \,   \frac{1}{\varphi(|B|^{\frac{1}{n}})}\dashint_{B}|u|\,\d x\colon \|u\|_{W^mX(\Omega)}\le1\right\}
                \\
            &\lesssim \, \sup_{|B|\le r} \frac{\widehat\varphi(|B|^{\frac{1}{n}})}{\varphi(|B|^{\frac{1}{n}})}
            = \sup_{s\le r} \, \frac{\widehat\varphi(s^{\frac{1}{n}})}{\varphi(s^{\frac{1}{n}})}
        \end{split}
    \end{equation} up to a constant   independent of $u$.
    By~\eqref{E:sobolev-to-morrey-vanishing-condition}, 
    \begin{equation}\label{E:thm1.4-limit}
        \lim_{r\to0^+} \sup_{s\le r}\, \frac{\widehat\varphi(s^{\frac{1}{n}})}{\varphi(s^{\frac{1}{n}})} =0.
    \end{equation}
    Coupling equation~\eqref{E:thm1.4-balls} with~\eqref{E:thm1.4-limit} yields equation~\eqref{E:morrey-vanishing-full}. The embedding~\eqref{fi2} is thus established.
\\
\eqref{fi2} $\Rightarrow$ \eqref{E:sobolev-to-morrey-vanishing-condition}. 
  We have to show that equation~\eqref{E:sobolev-to-morrey-vanishing-condition} follows from~\eqref{E:morrey-vanishing-full}.  To this purpose, we may assume, as in the proof of Theorem~\ref{T:2}, that $\Omega$ is the ball   centered at $0$  with $|\Omega|=1$. Minor modifications in the argument will show that the same proof applies to any domain $\Omega$.
     \\ Fix $r\in(0,1)$. Then
    \begin{equation}\label{E:thm-1.4-duality}
        \left\|s^{-1+\frac{m}{n}}\chi_{(r,1)}(s)\right\|_{X'(0,1)}
            = \sup\left\{\int_{r}^{1}s^{-1+\frac{m}{n}}f(s)\,\d s \colon \|f\|_{X(0,1)}\le1,f\ge0\right\}.
    \end{equation}
    To each nonnegative function $f$ satisfying $\|f\|_{X(0,1)}\le1$ we associate the function $v_f$ from~\eqref{E:new5}. Owing to equations~\eqref{E:thm-1.4-duality},~\eqref{E:Tmo2},~\eqref{E:T3} and~\eqref{E:psi-definition}, the following chain holds:
    \begin{align*}
        \frac{1}{\varphi(r^{\frac{1}{n}})}\left\|s^{-1+\frac{m}{n}}\chi_{(r,1)}(s)\right\|_{X'(0,1)}
        &= \frac{1}{\varphi(r^{\frac{1}{n}})}\sup\left\{\int_{r}^{1}s^{-1+\frac{m}{n}}f(s)\,\d s \colon
 \|f\|_{X(0,1)}\le1,f\ge0\right\}
            \\
        &\lesssim \frac{1}{\varphi(r^{\frac{1}{n}})}\sup\left\{\dashint_{B(r)}v_f \,\d x \colon \|f\|_{X(0,1)}\le1,f\ge0\right\}
            \\
        &\lesssim \frac{1}{\varphi(r^{\frac{1}{n}})}\sup\left\{\dashint_{B(r)}v_f \,\d x \colon \|v_f\|_{W^mX(\Omega)}\le1\right\}
            \\
        &\leq \sup\left\{\psi_{\varphi,v}(r) \colon \|v\|_{W^mX(\Omega)}\le1\right\},
    \end{align*}
up to multiplicative constants independent of $r$, where $B(r)$ is as in~\eqref{fi5}. By~\eqref{E:morrey-vanishing-full}, the expression on the right vanishes as $r\to0^+$. Hence, equation~\eqref{E:sobolev-to-morrey-vanishing-condition} follows.
\end{proof}

\section{Proofs of embeddings into Campanato spaces}\label{proof Campanato}

 Here, we provide proofs of the main results from Subsection \ref{S:campanato}.
  We begin with a preliminary lemma, which has a role in dealing with the statement {\it (iii)}~in Theorem~\ref{T:3}. Given a multi-index $\alpha=(\alpha_1, \cdots, \alpha_n)$, with $\alpha_i \in  \mathbb N \cup \{0\}$, we set $|\alpha|=  \sum_{i=1}^n \alpha_i$, and denote by 
$D^{\alpha}u$  the partial derivative of $u$ corresponding to $\alpha$.

\begin{lemma}\label{L:polynomial}
    Let $n\in \mathbb N$, $n\geq 2$ and $k\in \mathbb N \cup \{0\}$. If $H$ is a homogeneous harmonic polynomial  of degree $k+1$ in $\mathbb R^n$, then,
    \begin{equation} \label{E:orthoA}
        \int_{\mathbb S^{n-1}}  D^\alpha H  \,   {P}   \, \d \mathcal H^{n-1} =0
    \end{equation}
    for any  multi-index $\alpha$ satisfying $|\alpha| \leq k$ and any homogeneous polynomial $P$ of degree at most $k-|\alpha|$. Here, $\mathbb S^{n-1}$ stands for the $(n-1)$-dimensional unit sphere in $\mathbb R^n$.
\end{lemma}

\begin{proof}
    Let $k\in \mathbb N \cup \{0\}$ and let $H$ be a homogeneous harmonic polynomial of degree $k+1$ in $\mathbb R^n$. For any multi-index $\alpha$ such that  $|\alpha| \leq k$, $D^\alpha H$ is a homogeneous harmonic polynomial of degree $k+1-|\alpha|$. Hence, by~\cite[Lemma~2.2.1]{AG},  $D^\alpha H$ is orthogonal to any homogeneous polynomial $P$ whose degree does not exceed $k-|\alpha|$. Namely, equation~\eqref{E:orthoA} holds.
\end{proof}

\begin{proof}[Proof of Theorem~\ref{T:3}]   $(i) \Rightarrow (iii)$. For simplicity of notation, we
shall prove this implication  under the assumption that $\Omega$ is the ball   centered at
$0$  with $|\Omega|=1$. Thanks to Remark~\ref{rem-nov5},   it will be clear from the proof that an analogous argument
applies by replacing this ball with a sufficiently small ball contained in domain $\Omega$. 
\\
We first focus on the case when  $m<n+k+1$.  
Fix $k\in\{0,\dots,m-1\}$ and let $H$ be 
a homogenous harmonic polynomial of degree   $h \geq k+1$.
\\ For any nonnegative function $f \in X(0,1)$ such that $\operatorname{supp} f\subset (0,\tfrac 12)$,
define the function
    $w_f \colon \Omega \to\mathbb R$ as
    \begin{equation}\label{E:v}
        w_f(x)=
H(x)\int_{\omega_n|x|^n}^{1}f(s)s^{-m+\frac{m-h}{n}}(s-\omega_n|x|^n)^{m-1}\,\d s \quad \ \text{for $x\in \Omega$.}
    \end{equation}
    In the proof of this implication, the constants involved in the various relations only depend on $n$, $m$, $k$,  $h$ and $H$.
 Clearly,  $\supp w_f \subset B(\tfrac12)$, where $B(r)$ is defined as in~\eqref{fi5} for $r\in (0,1)$.
     Moreover, for any  multi-index
        $\alpha$  such that $|\alpha| \leq k$,  the derivative $D ^\alpha w_f(x)$ is a linear combination of terms of the form
    \begin{equation*}
    (D ^{\alpha_1} H)(x) P(x) g(|x|),
    \end{equation*}
    where $\alpha_1$ is a multi-index satisfying $\alpha_1 \leq \alpha$, $P$ is a homogeneous polynomial of degree $d$, with
     \begin{equation*}
    d \leq |\alpha| - |\alpha_1| \leq k- |\alpha_1|,
     \end{equation*}
    and $g\colon [0,\infty)\to \mathbb R$.
\\
    Fix any $r\in(0,1)$. Then, Lemma~\ref{L:polynomial} tells us that
\begin{align}\label{system}
    \int_{B(r)} (D^{\alpha_1} H)(x) P(x) g(|x|)\, \d x  = \int_{B(r)} (D^{\alpha_1} H)\left(\frac{x}{|x|}\right) P\left(\frac{x}{|x|}\right) |x|^{h-|\alpha_1|+d} g(|x|)\, \d x& \\ \nonumber
     = \int_0^{(\omega_n^{-1}r )^{1/n}} s^{h-|\alpha_1|+d+n-1} g(s)\int_{\mathbb S^{n-1}} D^{\alpha_1} H(\omega)  P(\omega) \, \d \mathcal H^{n-1}(\omega)\, \d s = 0.&
    \end{align}
Therefore,
\begin{align}\label{system1}
    \int_{B(r)} D^\alpha w_f\, \d x=0,
\end{align}
and, according to~\eqref{E:moments}, 
  \begin{equation}\label{E:zero-mean-polynomial}
    {P}^{k}_{B(r)}[w_f]   =0 \quad \ \text{for} \   r\in(0,1).
    \end{equation}
Coupling~\eqref{E:campanato-seminorm} with~\eqref{E:zero-mean-polynomial} tells us   that
\begin{equation}\label{E:campanato-seminorm-of-v}
    |w_f|_{\mathcal L^{k, \varphi(\cdot)}(\Omega)} \ \ge \ \sup_{r\in(0,1)} \, \frac{1}{\varphi(r^{\frac{1}{n}})r^{\frac{k}{n}}} \, \dashint_{B(r)}|w_f|\,\d x.
\end{equation}
We claim that
\begin{equation}\label{E:lower-estimate-for-v}
    \int_{r}^{\frac 12}s\sp{-1+\frac{m- h}{n}} f(s)\,\d s \ \lesssim\  r^{-\frac{h}n} \, \dashint_{B(r)}|w_f|\,\d x  \quad \ \text{for} \ r\in (0,\tfrac 12).
\end{equation}
To verify the claim, fix $r\in (0,\tfrac 12)$ and observe that, since $H$ is a homogeneous harmonic polynomial of degree $h$, there exist $\kappa >0$ and $\sigma >0$ and a set $G\subset \mathbb S^{n-1}$ such that $\mathcal H^{n-1}(G)>\sigma $ and $|H|\geq \kappa$ in $G$.
Let  
$E(r)=\{x\in B(r)\colon x/|x| \in G\}$.
Thereby,
\begin{align*}
    \dashint_{B(r)}|w_f|\,\d x &= \dashint_{B(r)}|H(x)|\int_{\omega_n |x|\sp n}\sp{1} s\sp{-m+\frac{m-h}{n}} f(s)(s-\omega_n |x|\sp n)^{m-1}\,\d s\,\d  x
        \\[1ex]
    &\ge \frac{\kappa}{r}\int_{E(r)}|x|^{h} \int_{\omega_n |x|\sp n}\sp{1} s\sp{-m+\frac{m-h}{n}} f(s) (s-\omega_n |x|\sp n)^{m-1}\,\d s \,\d  x
        \\[1ex]
    &\geq n\omega_n\frac{\kappa \sigma}{r}\int_0^{(\omega_n^{-1}r)^{1/n}}\rho^{h}\int_{\omega_n \rho\sp n}\sp{1}s\sp{-m+\frac{m-h}{n}} f(s) (s-\omega_n \rho\sp n)^{m-1}\,\d s
    \,\rho^{n-1}\,\d  \rho
        \\[1ex]
    &=   \omega_n^{-\frac{h}{n}}\frac{\kappa \sigma}{r} \int_0^{r} {\varrho}^{\frac{h}{n}}\int_{\varrho}\sp{1}  s\sp{-m+\frac{m-h}{n}} f(s)(s-\varrho)^{m-1}\,\d s\,\d \varrho.
\end{align*}
Next, by the Fubini theorem and  the fact that $f$ is supported in $(0,\tfrac12)$, the following chain holds:
\begin{align*}
  \frac{1}{r}\int_0^{r} {\varrho}^{\frac{h}{n}}\int_{\varrho}\sp{1}   s\sp{-m+\frac{m-h}{n}}  f(s) & (s-\varrho)^{m-1}\,\d s\,\d \varrho \\
&   \ge \frac{1}{r}\int_0^{r} {\varrho}^{\frac{h}{n}}\int_{r}\sp{1}s\sp{-m+\frac{m-h}{n}} f(s) (s-\varrho)^{m-1}\,\d s\,\d \varrho
        \\
    &=
    \frac{1}{r}\int_{r}^{1}s\sp{-m+\frac{m-h}{n}} f(s) \int_{0}\sp{r}\varrho^{\frac{h}n}(s-\varrho)^{m-1}\,\d\varrho \, \d  s
       \\[1ex]
    &   \geq  \frac{1}{r} \int_{r}^{1}s\sp{-m+\frac{m-h}{n}} f(s) \int_{0}\sp{\frac r2}\varrho^{\frac{h}n}(s-\varrho)^{m-1}\,\d\varrho \,\d s
      \\[1ex]
    &\gtrsim\frac{1}{r}\int_{r}^{1}s\sp{-1+\frac{m-h}{n}}  f(s) \int_{0}\sp{\frac r2}\varrho^{\frac{h}n}\,\d \varrho\,\d s
       \\[1ex]
    &\approx r^{\frac{h}n}\int_{r}^{1} s\sp{-1+\frac{m-h}{n}} f(s) \,\d s.
\end{align*}
 Let us next assume that either $k\leq m-2$ and $h=k+1$, or $k=m-1$ and $h \geq k+2$.
Our aim is to show that $w_f \in W^m_0X(\Omega)$
and
\begin{equation}\label{E:upper-estimate-for-v}
    \|w_f\|_{W^mX(\Omega)} \  \lesssim \  \|f\|_{X(0,1)}.
\end{equation}
We claim that, for every multi-index $\alpha$ such that $0\le|\alpha|\le  m-1$,  
\begin{equation}\label{E:claim-sublimiting}
    \begin{aligned}
|D^{\alpha}w_f(x)| \lesssim  \ \sum_{j=0}^{|\alpha|}& |x|^{nj-|\alpha|+h}\ \times
                \\
            &\times\int_{\omega_n|x|^n}^{1}f(s)s^{-m+\frac{m-h}{n}}(s-\omega_n|x|^n)^{m-j-1}\,\d s \quad \ \mbox{for a.e.} \ x \in
            \Omega,
    \end{aligned}
\end{equation}
whereas, for every $\alpha$ such that $|\alpha|= m$,  
\begin{equation}\label{E:claim-limiting}
    \begin{aligned}
        |D^{\alpha}w_f(x)| \lesssim & \sum_{j\in J} |x|^{nj-m+h}\   \times
            \\
        &\times\int_{\omega_n|x|^n}^{1}f(s)s^{-m+\frac{m-h}{n}}(s-\omega_n|x|^n)^{m-j-1}\,\d s
        +  f(\omega_n|x|^n)  \quad \ \mbox{for a.e.} \ x \in  \Omega,
    \end{aligned}
\end{equation}
 where $J=\{1,2,\dots,m-1\}$ if $k\leq m-2$ and $J=\{0,1,\dots,m-1\}$ if $k=m-1$.
\\ To prove this claim, we begin by observing 
 that, for every multi-index $\alpha$ satisfying $0\le|\alpha|\le m-1$, the derivative $D^{\alpha}w_f$ is a finite linear combination of terms of the form
\begin{equation}\label{E:term}
    x^{\beta}|x|^{a}\int_{\omega_n|x|^n}^{1}f(s)s^{-m+\frac{m-h}{n}}(s-\omega_n|x|^n)^{b}\,\d s,
\end{equation}
where $\beta$ is a multi-index and $a,b$ are nonnegative integers. Also, when $|\alpha|\le  m-2$, the power $b$ appearing in~\eqref{E:term} is strictly positive. The  triple $(\beta,a,b)$ will be called
the \emph{signature} of the   term in~\eqref{E:term} and we set $$\# (\beta,a,b) = |\beta|+a + nb.$$
Thus, the estimate~\eqref{E:claim-sublimiting} will follow if we   show that  each term appearing in the formula for $D ^{\alpha}w_f$,  with signature $(\beta,a,b)$, satisfies the  condition:
\begin{equation}\label{E:term-estimate}
    \#(\beta,a,b) \ge n(m-1)-|\alpha|+ h.
\end{equation}
We shall prove this fact by induction on $|\alpha|$.
If $|\alpha|=0$, then $D^{\alpha}w_f=w_f$, and $D^{\alpha}w_f$ only depends on  terms of the form
\begin{equation*}
    x^\beta\int_{\omega_n|x|^n}^{1}f(s)s^{-m+\frac{m-h}{n}}(s-\omega_n|x|^n)^{m-1}\,\d s,
\end{equation*}
whose signature $(\beta,0,m-1)$ satisfies the constraint $|\beta|=h$.
Since
\begin{equation*}
   \#(\beta,0,m-1) = h + n(m-1),
\end{equation*}
 the condition~\eqref{E:term-estimate} is fulfilled.
\\
Assume that the  condition~\eqref{E:term-estimate} holds for 
any term of the form~\eqref{E:term}
appearing in  $D^{\alpha} w_f$  for every multi-index  $\alpha$ such that $|\alpha|\le j$ for some $j\le m-2$.
Consider    any multi-index  $\gamma$ with  $|\gamma|=j+1$. Plainly, $D^{\gamma}w_f=\frac{\partial}{\partial x_i}(D^{\alpha}w_f)$, for some $i \in \{1, \dots,n\}$ and some multi-index $\alpha$, with $|\alpha|= j$.
Thereby,
\begin{align}\label{E:derivative-formula}
    \frac{\partial}{\partial x_i}\Big(x^{\beta}|x|^{a} \int_{\omega_n|x|^n}^{1}& f(s)s^{-m+\frac{m-h}{n}}(s-\omega_n|x|^n)^{b}\,\d s\Big)
        \\
    &=\  c_1 x^{\beta-e_i}|x|^{a}\int_{\omega_n|x|^n}^{1}f(s)s^{-m+\frac{m-h}{n}}(s-\omega_n|x|^n)^{b}\,\d s
    \nonumber
        \\
    & \qquad + c_2 x^{\beta+e_i}|x|^{a-2}\int_{\omega_n|x|^n}^{1}f(s)s^{-m+\frac{m-h}{n}}(s-\omega_n|x|^n)^{b}\,\d s
    \nonumber
        \\
    &\qquad  + c_3x^{\beta+e_i}|x|^{a+n-2}\int_{\omega_n|x|^n}^{1}f(s)s^{-m+\frac{m-h}{n}}(s-\omega_n|x|^n)^{b-1}\,\d s
    \nonumber
\end{align}
for suitable constants $c_1$, $c_2$ and $c_3$, and for a.e.~$x \in  \Omega$, in which the first summand is considered to be zero if $\beta_i=0$.  In both cases, the new terms have signatures $(\beta-e_i,a,b)$, $(\beta+e_i,a-2,b)$ and $(\beta+e_i,a+n-2,b-1)$, and
$$
      \#(\beta-e_i,a,b) = \#(\beta+e_i,a-2,b) =  \#(\beta+e_i,a+n-2,b-1) = |\beta|+a + nb -1.
$$
Hence, by the induction assumption,
\begin{align*}
     |\beta|+a + nb   \geq n(m-1)-(|\alpha|+1)+ h = n(m-1)- |\gamma| + h,
\end{align*}
namely, the condition~\eqref{E:term-estimate} holds for  each term appearing in $D^{\alpha}w_f$  with signature $(\beta,a,b)$, when $0\le|\alpha|\le m-1$. This establishes~\eqref{E:claim-sublimiting}.
\\ The estimate~\eqref{E:claim-limiting} can be proved similarly,   modulo the following two minor modifications. First, the exponent $b$ can be $0$ on 
 the left-hand side of~\eqref{E:derivative-formula}. In this case, the differentiation with respect to $x_i$ yields the first two terms on the right-hand side of~\eqref{E:derivative-formula} plus a third which is a multiple of 
\[
x^{\beta+e_i} |x|^{a+n+m-mn-h-2} f(\omega_n|x|^n).
\]
By the induction assumption~\eqref{E:term-estimate} applied with $|\alpha|=m-1$ and $b=0$, we obtain $|\beta|+a+m+n-mn-h-1 \geq 0$, as desired. Second,  when $k\leq m-2$ the term associated with $j=0$ does not appear on the right-hand side of the inequality~\eqref{E:claim-limiting}. Indeed, this term  has the form ~\eqref{E:term} with $b=m-1$. Such a term cannot be obtained by differentiating the function $w_f$ in~\eqref{E:v} $m$-times, since the degree of the polynomial $H$ is  less than $m$, and hence any $m$-th order derivative of $w_f$ involves at least one derivative of the integral in~\eqref{E:v}. As a consequence,  $b\leq m-2$ in this case.
\\
Now, the estimates~\eqref{E:claim-sublimiting} and~\eqref{E:claim-limiting} tell us that
   \begin{align}\label{E:nabla-m}
        |\nabla^{m}&w_f(x)| \\ &\lesssim \chi_{B(\frac12)}(x)   \left(\sum_{j\in J} |x|^{nj-m+h}\int_{\omega_n|x|^n}^{1}f(s)s^{-m+\frac{m-h}{n}}(s-\omega_n|x|^n)^{m-j-1}\,\d s
        + f(\omega_n|x|^n)\right) \nonumber
            \\ 
        &\le \chi_{B(\frac12)}(x)   \left(\sum_{j\in J} |x|^{nj-m+h}\int_{\omega_n|x|^n}^{1}f(s)s^{-j-1+\frac{m-h}{n}}\,\d s
        +  f(\omega_n|x|^n)\right) \quad \ \mbox{for a.e.} \ x \in \Omega. \nonumber
        \end{align}
       If $k\leq m-2$, then from equation \eqref{E:nabla-m} we deduce that
        \begin{align*}
        |\nabla^{m}w_f(x)|&\lesssim \chi_{B(\frac12)}(x)  \left(|x|^{n-m+k+1}\int_{\omega_n|x|^n}^{1}f(s)s^{-2+\frac{m-(k+1)}{n}}\,\d s
        +  f(\omega_n|x|^n) \right)
            \\
        &\lesssim \chi_{B(\frac12)}(x)  \left( (\omega_n|x|^n)^{1-\frac{m-(k+1)}{n}}\int_{\omega_n|x|^n}^{1}f(s)s^{-2+\frac{m-(k+1)}{n}}\,\d s
        +  f(\omega_n|x|^n)\right)   
    \end{align*}  for a.e.~$x \in \Omega$.
    Therefore, 
    \begin{align*}
\|\nabla^{m}w_f\|_{X(\Omega)} &\lesssim \left( \left\|r^{1-\frac{m-(k+1)}{n}}\int_{r}^{1}f(s)s^{-2+\frac{m-(k+1)}{n}}\,\d s\right\|_{X(0,1)}
        + \|f\|_{X(0,1)} \right).
    \end{align*}
 On the other hand, if $k=m-1$,   then equation \eqref{E:nabla-m} implies that
 \begin{align*}
|\nabla^{m}w_f(x)|&\lesssim   \chi_{B(\frac12)}(x)  \left(|x|^{-m+h}\int_{\omega_n|x|^n}^{1}f(s)s^{-1+\frac{m-h}{n}}\,\d s
        +  f(\omega_n|x|^n) \right)
            \\
        &\lesssim \chi_{B(\frac12)}(x)  \left( (\omega_n|x|^n)^{-\frac{m-h}{n}}\int_{\omega_n|x|^n}^{1}f(s)s^{-1+\frac{m-h}{n}}\,\d s
        +  f(\omega_n|x|^n)\right) \quad \ \mbox{for a.e.} \ x \in \Omega.
    \end{align*}
  An application  of Lemma~\ref{L:interpolation-lemma},  with $\gamma=\frac{m-(k+1)}{n}$ in the former case, and  with $\gamma=\frac{m-h+n}{n}$ in the latter case tells us that  
    \begin{align*}
\|\nabla^{m}w_f\|_{X(\Omega)}   \, \lesssim  \, \|f\|_{X(0,1)}.
    \end{align*}
   Thanks to~\eqref{Equ-Sob-norms}, this shows that $w_f \in W^m_0X(\Omega)$ and yields~\eqref{E:upper-estimate-for-v}.
\\
    Assume that  $k\in\{0,\dots,m-2\}$. From equations~\eqref{n.assoc.},~\eqref{E:lower-estimate-for-v},~\eqref{E:upper-estimate-for-v}, and~\eqref{E:campanato-seminorm-of-v} one can deduce that
    \begin{align*}
    &\sup_{r\in\left(0,\frac{1}{2}\right)}\frac{r^{\frac{1}{n}}}{\varphi(r^{\frac{1}{n}})}\left\|s^{-1+\frac{m-(k+1)}{n}}\chi_{\left(r,\frac 12\right)}(s) \right\|_{X'(0,1)}
            \\
        &\qquad =\sup_{r\in\left(0,\frac{1}{2}\right)}\frac{r^{\frac{1}{n}}}{\varphi(r^{\frac{1}{n}})}\sup\left\{\int_{r}^{\frac 12}s^{-1+\frac{m-(k+1)}{n}} f(s)\,\d s \colon f \in X_+(0,1),\ \supp f\subset (0,\tfrac12),\ \|f\|_{X(0,1)}\le1\right\}
            \\
        &\qquad \lesssim\sup\left\{\sup_{r\in\left(0,\frac{1}{2}\right)}\frac{1}{\varphi(r^{\frac{1}{n}})r^{\frac{k}{n}}}\dashint_{B(r)}|w_f|\,\d x \colon f \in X_+(0,1),\ \supp f\subset  (0,\tfrac12 ),\ \|f\|_{X(0,1)}\le1\right\}
            \\
        &\qquad \lesssim\sup\left\{\sup_{r\in\left(0,\frac{1}{2}\right)}\frac{1}{\varphi(r^{\frac{1}{n}})r^{\frac{k}{n}}}\dashint_{B(r)}|u|\,\d x\colon \|u\|_{W^mX(\Omega)}\le1\right\}
            \\
        &\qquad \lesssim\sup\left\{|u|_{\mathcal L^{k, \varphi(\cdot)}(\Omega)}\colon \|u\|_{W^mX(\Omega)}\le1\right\}.
    \end{align*}  
 An application of Lemma~\ref{L:standard} with $\theta(r)= {r^{-\frac{1}{n}}}{\varphi(r^{\frac{1}{n}})}$ and $\eta(r)=\left\|s^{-1+\frac{m-(k+1)}{n}}\chi_{\left(r,\frac 12\right)}(s) \right\|_{X'(0,1)}$ yields
    \begin{equation*}
\sup_{r\in(0,1)}\frac{r^{\frac{1}{n}}}{\varphi(r^{\frac{1}{n}})}\left\|s^{-1+\frac{m-(k+1)}{n}}\chi_{\left(r,\frac 12\right)}(s) \right\|_{X'(0,1)}
       \  \lesssim \  \sup\left\{|u|_{k, \mathcal L^{\varphi}(\Omega)} \colon \|u\|_{W^mX(\Omega)}\le1\right\}.
    \end{equation*}
    From Lemma~\ref{L:half-interval} we infer that
    \begin{equation*}
\sup_{r\in(0,1)}\frac{r^{\frac{1}{n}}}{\varphi(r^{\frac{1}{n}})}\left\|s^{-1+\frac{m-(k+1)}{n}}\chi_{(r,1)}(s) \right\|_{X'(0,1)}
        \ \lesssim \ \sup\left\{|u|_{\mathcal L^{k, \varphi}(\Omega)}\colon \|u\|_{W^mX(\Omega)}\le1\right\}.
    \end{equation*}
    Assumption {\it (i)} tells us that the supremum on the right-hand side is finite. Hence, equation~\eqref{E:sobolev-to-campanato-condition-subcritical} follows.
\\
  Now let $k=m-1$. 
  Set $n'=\frac n{n-1}$, the H\"older conjugate of $n$. Thanks to equations~\eqref{E:lower-estimate-for-v},~\eqref{E:campanato-seminorm-of-v},~\eqref{E:upper-estimate-for-v} and~\eqref{E:sobolev-to-campanato-embedding} one can deduce that
    \begin{align*}
        \sup_{r\in\left(0,\frac{1}{4}\right)}\frac{1}{\varphi(r^{\frac{1}{n}})r^{\frac{1}{n'}}}\left\|\chi_{(0,r)}\right\|_{X'(0,1)}
        &=\sup_{r\in\left(0,\frac{1}{4}\right)}\frac{r^{-1+\frac{1}{n}}}{\varphi(r^{\frac{1}{n}})}\left\|\chi_{(r,2r)}\right\|_{X'(0,1)}
            \\
        & \approx \sup_{r\in\left(0,\frac{1}{4}\right)}\frac{r^{\frac{1-m+h}{n}}}{\varphi(r^{\frac{1}{n}})}\left\|s^{-1 +\frac{m-h}{n}}\chi_{(r,2r)}(s)\right\|_{X'(0,1)} \\
        &
        \le \sup_{r\in\left(0,\frac{1}{4}\right)}\frac{r^{\frac{1-m+h}{n}}}{\varphi(r^{\frac{1}{n}})}\left\|s^{-1 +\frac{m-h}{n}}\chi_{\left(r,\frac 12\right)}(s)\right\|_{X'(0,1)}
            \\
        & = \sup_{r\in\left(0,\frac{1}{4}\right)}\frac{r^{\frac{1 -m+h}{n}}}{\varphi(r^{\frac{1}{n}})}\sup_{f \in X_+(0,1)}\frac{\int_{r}^{\frac 12}s^{-1+\frac{m-h}{n}} f(s)\,\d s}{\|f\|_{X(0,1)}}
            \\
        &\lesssim \sup_{f \in X_+(0,1)}\sup\limits_{r\in\left(0,\frac{1}{2}\right)}\frac{1}{\varphi(r^{\frac{1}{n}})r^{\frac{m-1}{n}}}\frac{\dashint_{B(r)}|w_f|\,\d x}{\|f\|_{X(0,1)}}
            \\[0.8ex]
        &\lesssim
        \sup_{f \in X_+(0,1)}\frac{|w_f|_{\mathcal L^{m-1, \varphi(\cdot)}(\Omega)}}{\|f\|_{X(0,1)}}
                   \lesssim \sup_{f \in X_+(0,1)}\frac{|w_f|_{\mathcal L^{m-1, \varphi(\cdot)}(\Omega)}}{\|w_f\|_{W^m X(\Omega)}}<\infty.
    \end{align*}
    Lemma~\ref{L:standard} with $\theta(r)= {r^{-\frac{1}{n}}}{\varphi(r^{\frac{1}{n}})}$ and $\eta(r)=r^{-1}{\left\|\chi_{(0,r)}\right\|_{X'(0,1)}}$ (a non-increasing function), and  the assumption~\eqref{E:admissible-new}  enable us to  establish~\eqref{E:sobolev-to-campanato-condition-critical}.
\\
    Next assume that $m\ge n+k+1$. Thus,  $k\in\{0,\dots,m-2\}$, and we only need to prove~\eqref{E:sobolev-to-campanato-condition-subcritical}.
    We claim that
    \begin{equation}\label{E:lower-bound-for-varphi}
        \sup_{r\in(0,1)}\frac{r}{\varphi(r)}<\infty.
    \end{equation}
  To verify this claim, notice that, trivially, any homogeneous harmonic polynomial $H$ of degree $k+1$ belongs to $W^mX(\Omega)$. Since we are assuming that the latter space is embedded into $\mathcal L^{k, \varphi(\cdot)}(\Omega)$, we have that $H\in \mathcal L^{k, \varphi(\cdot)}(\Omega)$ as well. Owing to equation~\eqref{E:orthoA} and to the homogeneity of the polynomial $H$,
$$\int_{B}D^\alpha H \, \d x=0$$
for every multi-index $\alpha$ such that $|\alpha|\leq k$ and   every ball $B$ centered at $0$. Thereby,
\begin{align}
    \label{nov140}
    P^k_{B}[H]=0.
\end{align}
As a consequence,
\begin{align}
    \label{nov141}
    \infty & >
    |H|_{\mathcal L^{k, \varphi(\cdot)}(\Omega)} = \sup_{B\subset \Omega} \, \frac{1}{\varphi(|B|^{\frac{1}{n}})|B|^{\frac{k}{n}}} \, \dashint_{B} \left|H -{P}^{k}_{B}[H]  \right|\,\d x\\  \nonumber & \geq
    \sup_{r\in (0,1)} \, \frac{1}{\varphi(r^{\frac 1n})r^\frac{k}{n}} \, \dashint_{B(r)}|H |\,\d x \ \approx  \ \sup_{r\in (0,1)} \, \frac{r^{k+1}}{\varphi(r)r^k} = \sup_{r\in (0,1)} \, \frac{r}{\varphi(r)},
\end{align}
up to constants depending on $n$, $k$, and $H$.
Property~\eqref{E:lower-bound-for-varphi} is thus established.
\\ Owing to~\eqref{E:lower-bound-for-varphi} and the fact that $-1+\frac{m-(k+1)}{n} \geq 0$, we have that
    \begin{equation*}
        \sup_{r \in (0,1)} \, \frac{r}{\varphi(r)}\left\|s^{-1+\frac{m-(k+1)}{n}}\chi_{{(r^n,1)}}(s)\right\|_{{X'(0,1)}}
       \  \lesssim \ \left\|\chi_{{(0,1)}}\right\|_{{X'(0,1)}}<\infty,
    \end{equation*}
    whence~\eqref{E:sobolev-to-campanato-condition-subcritical} follows.
\\
    $(iii)\Rightarrow(ii)$.  Unless otherwise stated,  in the proof of this implication  the constants involved in the various relations only depend on $n$, $m$, $k$ and $\varphi$. 
        Fix $k\in\{0,\dots,m-1\}$, and let $u \in W^mX(\Omega)$.
        Given a nonempty open ball $B$ such that $\overline B\subset\Omega$, set $r=|B|$.
    Owing to H\"older's inequality if $k\le n-2$, or trivially if $k\ge n-1$, one has that
    \begin{equation*}
        \dashint_{B} \left|u-{P}^{k}_{B}[u] \right|\,\d x \ \le \
            \begin{cases}
                r^{-1+\frac{k+1}{n}}     \left\|u-{P}^{k}_{B}[u]\right\|_{L^{\frac{n}{n-(k+1)}}(B)}   &\quad \text{if $k\le n-2$}
                    \\[1ex]
                  \|u-{P}^{k}_{B}[u]\|_{L^{\infty}(B)}&\quad \text{if $k\ge n-1$.}
            \end{cases}
    \end{equation*}
    This estimate, combined with either~\eqref{E:L1-1} if $k\le n-2$, or~\eqref{E:L1-2} if $k\ge n-1$, yields 
    \begin{equation}\label{E:upper-by-gradient}
        \dashint_{B} \left|u-{P}^{k}_{B}[u] \right|\,\d x \  \lesssim \
            r^{-1+\frac{k+1}{n}} \left\|\nabla^{k+1}u\right\|_{L^1(B)}.
    \end{equation} 
   Assume first that $m<n+k+1$ and $k\leq m-2$. Then, by equation~\eqref{E:upper-by-gradient} and  the inequality~\eqref{E:H}, 
    \begin{equation}\label{E:upper}
        \begin{split}
            \dashint_{B} \left|u-{P}^{k}_{B}[u] \right|\,\d x\  &\lesssim\  r^{-1+\frac{k+1}{n}} \left\|\nabla^{k+1}u\right\|_{X_{m-(k+1)}(\Omega)} \left\|\chi_{B}\right\|_{X'_{m-(k+1)}(\Omega)}
                \\[1ex]
            &  = r^{-1+\frac{k+1}{n}} \left\|\nabla^{k+1}u\right\|_{X_{m-(k+1)}(\Omega)} \left\|\chi_{(0,r)}\right\|_{X'_{m-(k+1)}(0,1)},
        \end{split}
    \end{equation}
    where $X_{m-(k+1)}(\Omega)$ is the norm defined as in 
\eqref{E:eucl_opt_norm-john}, with $m$ replaced with $m-(k+1)$. 
 
  \noindent    Since $|B|=r$, from the inequality~\eqref{E:upper}  one infers that
\begin{equation}\label{E:upper-divided}
        \begin{split}
            \frac{1}{\varphi(|B|^{\frac{1}{n}})|B|^{\frac{k}{n}}} \, \dashint_{B}\left|u-{P}^{k}_{B}[u] \right|\,\d x \
                \lesssim \ \frac{ r^{-1 +\frac{1}{n}}}{\varphi(r^{\frac{1}{n}})} \,\left\|\chi_{(0,r)}\right\|_{X'_{m-(k+1)}(0,1)}\left\|\nabla^{k+1}u\right\|_{X_{m-(k+1)}(\Omega)}.
        \end{split}
    \end{equation}
   Hence,
    \begin{equation}\label{E:upper-sups-prelim}
        |u|_{\mathcal L^{k, \varphi(\cdot)}(\Omega)}
        \ \lesssim \ \sup_{r\in(0,1)} \, \frac{r^{-1 +\frac{1}{n}}}{\varphi(r^{\frac{1}{n}})} \,\left\|\chi_{(0,r)}\right\|_{X'_{m-(k+1)}(0,1)}\left\|\nabla^{k+1}u\right\|_{X_{m-(k+1)}(\Omega)}.
    \end{equation}
        By Lemma~\ref{L:standard}, applied with $\theta(r)= {r^{-\frac{1}{n}}}{\varphi(r^{\frac{1}{n}})}$ and the non-increasing function $\eta(r)=r^{-1}{\left\|\chi_{(0,r)}\right\|_{X'(0,1)}}$, 
    \begin{equation*}
        \sup_{r\in(0,1)}\, \frac{r^{-1 +\frac{1}{n}}}{\varphi(r^{\frac{1}{n}})} \left\|\chi_{(0,r)}\right\|_{X'_{m-(k+1)}(0,1)}
       \  \lesssim \  \sup_{r\in(0,\frac 12)}\, \frac{r^{-1 +\frac{1}{n}}}{\varphi(r^{\frac{1}{n}})} \left\|\chi_{(0,r)}\right\|_{X'_{m-(k+1)}(0,1)}.
    \end{equation*} 
    Notice that here we have also made use of the assumption \eqref{E:admissible-new}.
    Coupling the latter inequality with the inequality~\eqref{E:upper-sups-prelim} yields:
    \begin{equation}\label{E:upper-sups}
        |u|_{\mathcal L^{k, \varphi(\cdot)}(\Omega)}
        \ \lesssim \ \sup_{r\in\left(0,\frac{1}{2}\right)}\, \frac{r^{-1 +\frac{1}{n}}}{\varphi(r^{\frac{1}{n}})} \left\|\chi_{(0,r)}\right\|_{X'_{m-(k+1)}(0,1)}\left\|\nabla^{k+1}u\right\|_{X_{m-(k+1)}(\Omega)}.
    \end{equation}
    Thanks to Lemma~\ref{L:interval-norm} applied   with   $\alpha= {n^{-1}}(m-k-1)$,  
     \begin{equation}\label{E:fundamental-estimate}
        \begin{split}
            \left\|\chi_{(0,r)}\right\|_{X'_{m-(k+1)}(0,1)}&=\|s^{\frac{m-k-1}{n}}\chi^{**}_{(0,r)}(s)\|_{X'(0,1)} \ 
            \lesssim \  r \left\|s^{-1+\frac{m-k-1}{n}}  \chi_{(r,1)}(s)\right\|_{X'(0,1)} 
        \end{split}
    \end{equation}
    for   $r\in\left(0,\tfrac{1}{2}\right)$.
    By~\cite[Theorem~6.2]{CPS},
    \begin{equation*}
        W^{m-(k+1)}X(\Omega) \to X_{m-(k+1)}(\Omega).
    \end{equation*}
    Thus,
    \begin{equation}\label{E:iterated-embedding}
        \left\|\nabla^{k+1}u\right\|_{X_{m-(k+1)}(\Omega)} \ \lesssim \ \sum_{j=0}^{m-(k+1)}\|\nabla^j (\nabla^{k+1}u)\|_{X(\Omega)}
       \  = \ \sum_{j=k+1}^{m}\|\nabla^j u \|_{X(\Omega)},
    \end{equation}
    up to a constant    depending  on $m, k, \|\cdot \|_{X(0,1)}$, and $\Omega$. Merging~\eqref{E:fundamental-estimate} and~\eqref{E:iterated-embedding}     into~\eqref{E:upper-sups} results in:
    \begin{align*}
            |u|_{\mathcal L^{k, \varphi(\cdot)}(\Omega)}
            &\lesssim\sup_{r\in\left(0,\frac{1}{2}\right)}\, \frac{r^{\frac{1}{n}}}{\varphi(r^{\frac{1}{n}})}\left\|s^{-1+\frac{m-k-1}{n}}\chi_{(r,1)}(s)\right\|_{X'(0,1)} \sum_{j=k+1}^{m}\|\nabla^j u \|_{X(\Omega)}
                \\
            &\le \sup_{r\in(0,1)}\, \frac{r^{\frac{1}{n}}}{\varphi(r^{\frac{1}{n}})}\left\|s^{-1+\frac{m-k-1}{n}}\chi_{(r,1)}(s)\right\|_{X'(0,1)} \sum_{j=k+1}^{m}\|\nabla^j u \|_{X(\Omega)}
    \end{align*}
    up to a constant    depending   on $n, m, k, \varphi,  \|\cdot \|_{X(0,1)}$,  and $\Omega$.
    Consequently, the   inequality~\eqref{E:sobolev-to-campanato-inequality} follows from \eqref{E:sobolev-to-campanato-condition-subcritical}.
\\  Next,  assume that  $k=m-1$. From equation \eqref{E:upper-by-gradient} and the inequality
\eqref{E:H} we analogously deduce that
    \begin{equation}\label{E:upper-sups-critical}
        |u|_{\mathcal L^{k, \varphi(\cdot)}(\Omega)}
        \ \lesssim\ \sup_{r\in(0,1)}\, \frac{1}{\varphi(r^{\frac{1}{n}})r^{\frac{1}{n'}}} \left\|\chi_{(0,r)}\right\|_{X'(0,1)}\left\|\nabla^{m}u\right\|_{X(\Omega)}.
    \end{equation} 
    Hence, if ~\eqref{E:sobolev-to-campanato-condition-critical} is in force, then inequality~\eqref{E:sobolev-to-campanato-inequality} holds.
\\ 
    Finally, consider the case when $m\ge n+k+1$. Then~\eqref{E:sobolev-to-campanato-condition-subcritical} reduces to~\eqref{E:lower-bound-for-varphi}.
    From~\eqref{E:upper-by-gradient} we obtain that
    \begin{equation*}
        \dashint_{B} \left|u-{P}^{k}_{B}[u] \right| \,\d x \ \lesssim \ r^{-1+\frac{k+1}{n}} \left\|\nabla^{k+1}u\right\|_{L^{\infty}(B)} \left\|\chi_{B}\right\|_{L^1(B)}
               = r^{\frac{k+1}{n}} \left\|\nabla^{k+1}u\right\|_{L^{\infty}(B)}.
    \end{equation*} 
    From the latter inequality and~\eqref{E:lower-bound-for-varphi} one has that
    \begin{equation*}
         \frac{1}{\varphi(|B|^{\frac{1}{n}})|B|^{\frac{k}{n}}}\dashint_{B} \left|u-{P}^{k}_{B}[u] \right| \,\d x
        \  \lesssim \ \frac{r^{\frac{1}{n}}}{\varphi(r^{\frac{1}{n}})}\left\|\nabla^{k+1}u\right\|_{L^{\infty}(B)}
         \lesssim \left\|\nabla^{k+1}u\right\|_{L^{\infty}(B)}.
    \end{equation*}
    As a consequence,
    \begin{equation*}
         |u|_{\mathcal L^{k, \varphi(\cdot)}(\Omega)}
         \ \lesssim \ \left\|\nabla^{k+1}u\right\|_{L^{\infty}(\Omega)}.
    \end{equation*}
     Owing to the assumption $m\ge n+k+1$ and to the second embedding in~\eqref{E:imm},  
    \begin{equation*}
        W^{m-(k+1)}X(\Omega)\to W^{n}L^1(\Omega)\to L^{\infty}(\Omega).
    \end{equation*}
   Hence,
    \begin{equation*}
         |u|_{\mathcal L^{k, \varphi(\cdot)}(\Omega)} \  \lesssim \  \sum_{j=0}^{n}\left\|\nabla^j(\nabla^{k+1}u)\right\|_{L^{1}(\Omega)}
         \le \sum_{j=k+1}^{n+k+1}\left\|\nabla^ju\right\|_{L^{1}(\Omega)}
         \  \lesssim \  \sum_{j=k+1}^{m}\left\|\nabla^ju\right\|_{X(\Omega)},
    \end{equation*}
    namely, equation \eqref{E:sobolev-to-campanato-inequality}.
\\
    $(ii)\Rightarrow(i)$.  This implication is trivial. 
\end{proof}

\begin{proof}[Proof of Theorem~\ref{T:campanato-optimal-range}]  The function  $\overline\varphi$ defined by~\eqref{E:optimal-campanato-range-k} is easily seen to satisfy the condition \eqref{E:admissible-new}.
Moreover, trivially,
    \begin{equation*}
        \sup_{r\in(0,1)}\, \frac{r} {\overline\varphi(r)}\left\|s^{-1+\frac{m-(k+1)}{n}}\chi_{{(r^n,1)}}(s)\right\|_{{X'(0,1)}}=1
        \quad \ \text{if $k\in\{0,\dots,m-2\}$}
    \end{equation*}
    and
    \begin{equation*}
        \sup_{r\in(0,1)}\, \frac{1}{\overline\varphi(r)r^{n-1}}\left\|\chi_{{(0,r^n)}}\right\|_{{X'(0,1)}}=1
        \quad \ \text{if $k=m-1$.}
    \end{equation*}  The embedding~\eqref{E:optimal-campanato-range-embedding} hence follows from Theorem~\ref{T:3}.
\\
  As far as the optimality of the space $\mathcal L^{k, \overline\varphi(\cdot)}(\Omega)$ 
 is concerned, assume that the embedding~\eqref{E:sobolev-to-campanato-embedding}
  holds for some admissible function $\varphi\colon(0,\infty)\to(0,\infty)$. We have to show that
\begin{equation}\label{E:OPT}\mathcal L^{k, \overline\varphi(\cdot)}(\Omega) \to \mathcal L^{k, \varphi(\cdot)}(\Omega).\end{equation}
By Theorem~\ref{T:3}, the embedding~\eqref{E:sobolev-to-campanato-embedding} implies that either $k\in\{0,\dots,m-1\}$ and~\eqref{E:sobolev-to-campanato-condition-subcritical} holds, or $k=m-1$ and~\eqref{E:sobolev-to-campanato-condition-critical} holds.     Hence, according to   the definition of $\overline\varphi$,  in either case  one has that
    \begin{equation*}\label{E:comparison}
       \sup_{r\in(0,1)}\frac{\overline\varphi(r)}{\varphi(r)}<\infty,
    \end{equation*}
    which entails the  embedding~\eqref{E:OPT}.
\end{proof}

\begin{proof}[Proof of Theorem~\ref{T:campanato-optimal-domain}]   Let us first focus on the case when
  $m\leq n+k$. The fact that $\|\cdot\|_{\widehat X(0,1)}$ is a rearrangement-invariant function norm can be readily verified.  Observe that
    \begin{equation*}
        \|\chi_{(0,1)}\|_{\widehat X(0,1)}        
          \  \approx \ \sup\limits_{r\in(0,1)}\frac{r^{\frac{1}{n}}\left(1-r^{\frac{m-(k+1)}{n}}\right)}{\varphi(r^{\frac{1}{n}})} \le \sup\limits_{r\in(0,1)}\frac{r^{\frac{1}{n}}}{\varphi(r^{\frac{1}{n}})}
          \quad \ \text{if} \   k\in\{0,\dots,m-2\} 
                     \end{equation*}
and
    \begin{equation*}        \|\chi_{(0,1)}\|_{\widehat X(0,1)}=\sup\limits_{r\in(0,1)}\frac{r^{\frac{1}{n}}}{\varphi(r^{\frac{1}{n}})}
          \quad \ \text{if} \  k=m-1,
           \end{equation*}    hence property~\eqref{N4}   follows from~\eqref{E:Campanato-condition-on-phi} and \eqref{E:admissible-new}.
           \\
           New, let $\|\cdot\|_{X(0,1)}$ be  any rearrangement-invariant function norm   and set
    \begin{equation}\label{Bk}
    B(k,\varphi) =
        \begin{cases}
            \sup\limits_{r\in(0,1)}\frac{r^{\frac{1}{n}}}{\varphi(r^{\frac{1}{n}})}\left\|s^{-1+\frac{m-(k+1)}{n}}\chi_{{(r,1)}}(s)\right\|_{{X'(0,1)}}
            &\quad \text{if} \ k\in\{0,\dots,m-2\}  
                \\[4ex]
    \sup\limits_{r\in(0,1)}\frac{1}{\varphi(r^{\frac1n})r^{\frac{1}{n'}}}\left\|\chi_{{(0,r)}}\right\|_{{X'(0,1)}}
            &\quad \text{if} \  k=m-1.
         \end{cases}
    \end{equation}
    We claim that
    \begin{equation}\label{E:star-equivalence}
        B(k,\varphi) \ \approx \ \sup_{\|f\|_{X(0,1)}\le1}\|f\|_{\widehat X(0,1)} 
    \end{equation}  for every $k\in\{0,\dots,m-1\}$. Here, and in the rest of this proof, the equivalence holds up to 
      constants   depending possibly   on $n,m,k$ and $\varphi$.
     \\
 To verify equation~\eqref{Bk}, fix $r\in \left(0,\frac14\right)$.     If $k\in\{0,\dots,m-2\}$, then   
 thanks to~\eqref{n.assoc.2}  and Lemma~\ref{L:double-integral} (applied with $\alpha= {n^{-1}}(m-k-1)$), the following chain holds:
    \begin{align*}
        \|s^{-1+\frac{m-(k+1)}{n}}\chi_{(r,1)}(s)\|_{X'(0,1)} &= \sup_{\|f\|_{X(0,1)}\le1} \int_{0}^{1}\left[(\cdot)^{-1+\frac{m-(k+1)}{n}}\chi_{(r,1)}(\cdot)\right]^*(s) f^{*}(s)\, \d s
            \\
        &= \sup_{\|f\|_{X(0,1)}\le1} \int_{0}^{1}(r+s)^{-1+\frac{m-(k+1)}{n}} \chi_{(0,1-r)}(s) f^{*}(s)\, \d s
            \\
        &\approx \sup_{\|f\|_{X(0,1)}\le1}\left(r^{-1+\frac{m-(k+1)}{n}}\int_{0}^{r}f^{*}(s)\,\d s+\int_{r}^{1-r}s^{-1+\frac{m-(k+1)}{n}}f^{*}(s)\,\d s\right)
            \\
        & \approx \sup_{\|f\|_{X(0,1)}\le1}\int_{r}^{1}s^{-1+\frac{m-(k+1)}{n}}f^{**}(s)\, \d s .
    \end{align*} 
   If $k=m-1$, it follows   from~\eqref{n.assoc.2} that
    \begin{equation*}
    \left\|\chi_{(0,r)}\right\|_{X'(0,1)} =  \sup_{\|f\|_{X(0,1)}\le1}\int_{0}^{r}f^{*}(s)\, \d s.
    \end{equation*}
    Therefore,
    \begin{equation*}         \sup_{r\in\left(0,\frac14\right)}\frac{r^{\frac{1}{n}}}{\varphi(r^{\frac{1}{n}})}\left\|s^{-1+\frac{m-(k+1)}{n}}\chi_{(r,1)}(s)\right\|_{X'(0,1)} \ \approx \ \sup_{\|f\|_{X(0,1)}\le1}\sup_{r\in\left(0,\frac14\right)}\frac{r^{\frac{1}{n}}}{\varphi(r^{\frac{1}{n}})}\int_{r}^{1}s^{-1+\frac{m-(k+1)}{n}}f^{**}(s)\, \d s
    \end{equation*}
    if $k\in\{0,\dots,m-2\}$, and
    \begin{equation*}
    \sup_{r\in\left(0,\frac14\right)}\, \frac{1}{\varphi(r^{\frac1n})r^{\frac{1}{n'}}}\left\|\chi_{(0,r)}\right\|_{X'(0,1)} = \sup_{\|f\|_{X(0,1)}\le1}\sup_{r\in\left(0,\frac14\right)}\, \frac{r^{\frac{1}{n}}}{\varphi(r^{\frac1n})}f^{**}(r)
    \end{equation*}
    if $k=m-1$. Hence    one can conclude via  Lemma~\ref{L:standard} that property~\eqref{E:star-equivalence} holds.
       \\
    The choice $\|\cdot\|_{X(0,1)}=\|\cdot\|_{\widehat X(0,1)}$ in~\eqref{E:star-equivalence} yields, via an application of  Theorem~\ref{T:3}, that the  embedding~\eqref{E:optimal-campanato-domain-embedding} holds.  
      Moreover, if the embedding~\eqref{E:sobolev-to-campanato-embedding}
  holds for some rearrangement-invariant function norm $\|\cdot\|_{X(0,1)}$, then Theorem~\ref{T:3} tells us that $B(k,\varphi)<\infty$ and, according to~\eqref{E:star-equivalence}, 
    \begin{equation*}
        \sup_{\|f\|_{X(0,1)}\le1} \|f\|_{\widehat X(0,1)}<\infty.
    \end{equation*}
   Thus,   $X (0,1) \to \widehat X (0,1)$, and this proves that  $\widehat X(\Omega)$ is the optimal (largest possible) rearrangement-invariant space in~\eqref{E:optimal-campanato-domain-embedding}.
\\ 
As for the case $m\ge n+k+1$,  thanks to Theorem~\ref{T:3},
    \begin{equation*}
        W^mL^{1}(\Omega) \to \mathcal L^{k, \varphi_1(\cdot)}(\Omega),
    \end{equation*}
     where $\varphi_{1}(r)=r$ for $r\in(0,\infty)$. If $\varphi$ is an admissible function  fulfilling~\eqref{E:Campanato-condition-on-phi}, then      clearly   $\varphi_1 \lesssim \varphi$ on $(0,r_0)$, for some $r_0 \in (0,1)$.  Consequently, $\mathcal L^{k, \varphi_1(\cdot)}(\Omega)\to\mathcal L^{k, \varphi(\cdot)}(\Omega)$, whence
    \begin{equation*}
        W^mL^{1}(\Omega) \to \mathcal L^{k, \varphi(\cdot)}(\Omega).
    \end{equation*}
    The optimality of $L^1(\Omega)$  in~\eqref{E:optimal-campanato-domain-embedding} follows from the second embedding in~\eqref{E:imm}.
\\
Finally, assume that the condition~\eqref{E:Campanato-condition-on-phi} is not satisfied.   Let $\|\cdot\|_{X(0,1)}=\|\cdot\|_{L^{\infty}(0,1)}$. Then $B(k,\varphi)=\infty$, and from Theorem~\ref{T:3} and \eqref{Bk} one infers that the embedding
    \begin{equation*}
        W^mL^{\infty}(\Omega) \to \mathcal L^{k, \varphi(\cdot)}(\Omega)
    \end{equation*}
    does not hold. Hence, the fact that any rearrangement-invariant space $X(\Omega)$  satisfies~\eqref{E:imm} tells us that the embedding~\eqref{E:sobolev-to-campanato-embedding}  fails for every rearrangement-invariant space $X(\Omega)$.  
\end{proof}

\begin{proof}[Proof of Theorem~\ref{T:sobolev-campanato-vanishing}]
  $(i) \Rightarrow (iii)$.   
    Assume first that $m-k\le n$.
    As in the proof of analogous implications above, we assume, for simplicity of notation, that   $\Omega$ is the ball  centered at $0$  with $|\Omega|=1$.
    Fix any  homogeneous  harmonic polynomial $H$ of degree $k+1$. With each  measurable function $f \colon (0, \infty) \to [0, \infty)$, such that  $\supp f \subset \left(0,\tfrac12\right)$, we associate the function $w_f$ defined as in~\eqref{E:v}. Fix $r\in\left(0,\tfrac{1}{4}\right)$ and let $B(r)$ be the ball obeying~\eqref{fi5}.
\\ If  $k\in\{0,\dots,m-2\}$,  then
    \begin{equation}\label{E:SCV1}
        \begin{split}
            &\left\|s^{-1+\frac{m-(k+1)}{n}}\chi_{\left(r,\frac{1}{2}\right)}(s)\right\|_{X'(0,1)}
                \\
            &\qquad \qquad= \sup\left\{\int_{r}^{\frac{1}{2}}s^{-1+\frac{m-(k+1)}{n}}f(s)\, \d s\colon  f\ge0,\ \supp f\subset\left(0,\tfrac12\right),\ \|f\|_{X(0,1)}\le1\right\}.
        \end{split}
    \end{equation}
     Thus, by~\eqref{E:SCV1}, and~\eqref{E:lower-estimate-for-v}, 
    \begin{align*}
         \frac{r^{\frac{1}{n}}}{\varphi(r^{\frac{1}{n}})}&\left\|s^{-1+\frac{m-(k+1)}{n}}\chi_{\left(r,\frac{1}{2}\right)}(s)\right\|_{X'(0,1)}
           \\
         & \lesssim \ \frac{r^{\frac{1}{n}}}{\varphi(r^{\frac{1}{n}})}\sup\left\{r^{-\frac{k+1}{n}}\dashint_{B(r)}|w_f| \, \d x\colon  f\ge0,\ \supp f\subset\left(0,\tfrac12\right),\ \|f\|_{X(0,1)}\le1\right\}
           \\
         & = \sup\left\{\frac{1}{\varphi(r^{\frac{1}{n}})r^{\frac{k}{n}}}\dashint_{B(r)}|w_f|\, \d x\colon  f\ge0,\ \supp f\subset\left(0,\tfrac12\right),\ \|f\|_{X(0,1)}\le1\right\}.
    \end{align*}
 The inequality above holds up to a constant depending only on $n$, $m$, $k$ and $H$. The same dependence holds for the various relations 
 throughout the proof of the current part of the proof. 
\\
    Owing to~\eqref{E:upper-estimate-for-v}  and~\eqref{E:campanato-seminorm-of-v},  
       \begin{align*}
        \frac{r^{\frac{1}{n}}}{\varphi(r^{\frac{1}{n}})}&\left\|s^{-1+\frac{m-(k+1)}{n}}\chi_{\left(r,\frac{1}{2}\right)}(s)\right\|_{X'(0,1)}
          \\
        & \lesssim \sup\left\{\frac{1}{\varphi(r^{\frac{1}{n}})r^{\frac{k}{n}}}\dashint_{B(r)}|w_f|\, \d x\colon  \|w_f\|_{W^mX(\Omega)}\le1\right\}
          \\
        & \lesssim \sup\left\{\sup_{B\subset \Omega,\,|B|\le r}\frac{1}{\varphi(|B|^{\frac{1}{n}})|B|^{\frac{k}{n}}}\dashint_{B} \left|w_f-P^k_{B}[w_f]\right|\, \d x\colon  \|w_f\|_{W^mX(\Omega)}\le1\right\}
          \\
        & = \sup\left\{\varrho_{\varphi,k,w_f}(r)\colon  \|w_f\|_{W^mX(\Omega)}\le1\right\}
          \\
        & \le \sup\left\{\varrho_{\varphi,k,u}(r)\colon  \|u\|_{W^mX(\Omega)}\le1\right\}.
    \end{align*} 
    Thanks to Lemma~\ref{L:half-interval},  
    \begin{equation*}
        \frac{r^{\frac{1}{n}}}{\varphi(r^{\frac{1}{n}})}\left\|s^{-1+\frac{m-(k+1)}{n}}\chi_{(r,1)}(s)\right\|_{X'(0,1)}
           \  \lesssim \ \sup\left\{\varrho_{\varphi,k,u}(r)\colon  \|u\|_{W^mX(\Omega)}\le1\right\}.
    \end{equation*}
    By~\eqref{E:vanishing-full},
    \begin{equation*}
        \lim_{r\to0^+}\sup\left\{\varrho_{\varphi,k,u}(r)\colon  \|u\|_{W^mX(\Omega)}\le1\right\}=0.
    \end{equation*}
    Altogether, equation~\eqref{E:sobolev-to-campanato-vanishing-condition-subcritical} follows.
\\ If $k = m-1$, then, by~\eqref{E:lower-estimate-for-v} and~\eqref{E:upper-estimate-for-v} with $h\geq k+2$,
    \begin{align*}
    \frac{1}{\varphi(r^{\frac1n})r^{\frac{1}{n'}}}\|\chi_{(0,r)}\|_{X'(0,1)} &= \frac{1}{\varphi(r^{\frac1n})r^{\frac{1}{n'}}}\|\chi_{(r,2r)}\|_{X'(0,1)}
    \ \lesssim \ \frac{r^{\frac{1-m+h}{n}}}{\varphi(r^{\frac1n})}\|s^{-1+\frac{m-h}{n}}\chi_{(r,2r)}(s)\|_{X'(0,1)}
        \\
    & \leq \frac{r^{\frac{1-m+h}{n}}}{\varphi(r^{\frac1n})}
        \sup\left\{\int_{r}^{\frac12}f(s) s^{-1+\frac{m-h}{n}}\, \d s\colon  f\ge0,\ \|f\|_{X(0,1)}\le 1,\ \supp f\subset\left(0,\tfrac12\right)\right\}
        \\
    &\lesssim \frac{1}{\varphi(r^{\frac1n})r^{\frac{m-1}{n}}}
        \sup\left\{\dashint_{B(r)}|w_f|\, \d x\colon  f\ge0,\ \|f\|_{X(0,1)}\le 1,\ \supp f\subset\left(0,\tfrac12\right)\right\}
        \\
    &\lesssim \frac{1}{\varphi(r^{\frac1n})r^{\frac{m-1}{n}}}
        \sup\left\{\dashint_{B(r)} \left|u-P^k_{B(r)}[u]\right|\, \d x\colon  \|u\|_{W^mX(\Omega)}\le 1\right\}
        \\
    &\le \sup\left\{\varrho_{\varphi,m-1,u}(r)\colon  \|u\|_{W^mX(\Omega)}\le 1\right\}.
    \end{align*}
    Equation~\eqref{E:vanishing-full} entails that
    \begin{equation*}
        \lim_{r\to0^+}\sup\left\{\varrho_{\varphi,m-1,u}(r)\colon \|u\|_{W^mX(\Omega)}\le 1\right\}=0.
    \end{equation*}
    Consequently,~\eqref{E:sobolev-to-campanato-vanishing-condition-critical} follows.
    \\
    Now assume that $m\ge n+ k+ 1$. We have to verify~\eqref{E:sobolev-to-campanato-vanishing-condition-subcritical}. We  claim that 
    \begin{equation}\label{E:limit-of-the-basic-fraction}
        \lim_{r\to0^+}\frac{r}{\varphi(r)}=0.
    \end{equation}
    Fix any homogeneous harmonic polynomial $H$ of degree $k+1$. Since $H\in W^mX(\Omega)\subset V\!\mathcal L^{k, \varphi(\cdot)}(\Omega)$, via~\eqref{nov140} one infers that
    \begin{equation*}
        \lim_{r\to 0^+}\sup_{B\subset \Omega,|B|\le r}\frac{1}{\varphi(|B|^{\frac{1}{n}})|B|^{\frac{k}{n}}} \dashint_{B}|H|\,\d x=0.
    \end{equation*}
    Therefore,
    \begin{align*}
    0 &=\lim_{r\to 0^+}\sup_{B\subset \Omega,|B|\le r} \, \frac{1}{\varphi(|B|^{\frac{1}{n}})|B|^{\frac{k}{n}}} \, \dashint_{B}|H|\,\d x
        \\  
    & \ge
    \lim_{r\to 0^+}\frac{1}{\varphi(r)r^k}\dashint_{B(r)}|H|\,\d x \approx  \lim_{r\to 0^+}\frac{r^{k+1}}{\varphi(r)r^k} = \lim_{r\to 0^+}\frac{r}{\varphi(r)},
    \end{align*}
    up to constants depending on $n$, $k$ and $H$.  Hence, 
   equation~\eqref{E:limit-of-the-basic-fraction} is established.
    Owing   to~\eqref{E:limit-of-the-basic-fraction} and the fact that $-1+\frac{m-(k+1)}{n}\geq 0$, we have that
    \begin{equation*}
        \lim_{r\to 0^+}\frac{r^{\frac 1n}}{\varphi(r^{\frac1n})}\left\|s^{-1+\frac{m-(k+1)}{n}}\chi_{{(r,1)}}(s)\right\|_{{X'(0,1)}}
        \leq\lim_{r\to 0^+}\frac{r^{\frac 1n}}{\varphi(r^{\frac1n})}\left\|\chi_{{(0,1)}}\right\|_{{X'(0,1)}}=0,
    \end{equation*}
    and~\eqref{E:sobolev-to-campanato-vanishing-condition-subcritical} follows.
\\
    $(iii)\Rightarrow(ii)$. Let $\overline\varphi$ be the function from~\eqref{E:optimal-campanato-range-k} and~\eqref{E:optimal-campanato-range-m-1}. By Theorem~\ref{T:campanato-optimal-range},  
    \begin{equation*}
        W^mX(\Omega)\to \mathcal L^{k, \overline\varphi(\cdot)}(\Omega).
    \end{equation*}
    On the other hand, Theorem~\ref{T:3} ensures that
    \begin{equation*}
        |u|_{\mathcal L^{k, \overline\varphi(\cdot)}(\Omega)}\  \lesssim \ \sum_{j=k+1}^{m}\|\nabla^j u \|_{X(\Omega)}
    \end{equation*}
    for $u \in W^mX(\Omega)$.
    Owing to~\eqref{E:campanato-seminorm}, the latter inequality yields
    \begin{equation*}
        \sup_{B\subset\Omega} \, \frac{1}{\overline\varphi(|B|^{\frac{1}{n}})|B|^{\frac{k}{n}}} \, \dashint_{B} \left|u-{P}^{k}_{B}[u] \right|\, \d  x \ \lesssim \  \sum_{j=k+1}^{m}\|\nabla^j u \|_{X(\Omega)}
    \end{equation*}
     for $u \in W^mX(\Omega)$.
    Hence, given any function $u$ such that $\sum_{j=k+1}^{m}\|\nabla^j u \|_{X(\Omega)}\le1$ and any ball $B\subset \Omega$, one has that
    \begin{equation*}
        \frac{1}{|B|^{\frac{k}{n}}} \, \dashint_{B}\left|u-{P}^{k}_{B}[u] \right|\, \d  x \ \lesssim \  \overline\varphi(|B|^{\frac{1}{n}}),
    \end{equation*}
    up to a constant independent of $B$.
    Fix $r\in(0,1)$. Let $k\in\{0,\dots,m-1\}$. Therefore,
    \begin{align}\label{feb200}
   & \sup\left\{\varrho_{\varphi, k, u}(r) \colon \sum_{j=k+1}^m\|\nabla^j u\|_{X(\Omega)}\le 1\right\}
        \\ \nonumber
    &\qquad \qquad \qquad  = \sup\left\{\sup_{B\subset\Omega,\ |B|\le r} \, \frac{1}{\varphi(|B|^{\frac{1}{n}})|B|^{\frac{k}{n}}} \, \dashint_{B}\left|u-{P}^{k}_{B}[u] \right|\, \d  x\colon \sum_{j=k+1}^m\|\nabla^j u\|_{X(\Omega)}\le 1\right\}
        \\ \nonumber 
    &\qquad \qquad \qquad \le \sup_{B\subset\Omega,\ |B|\le r} \,\frac{\overline\varphi(|B|^{\frac1n})}{\varphi(|B|^{\frac1n})} \le \sup_{s\in(0,r)} \, \frac{\overline\varphi(s^{\frac1n})}{\varphi(s^{\frac1n})}.
    \end{align}
    Coupling~\eqref{feb200} with 
   either~\eqref{E:sobolev-to-campanato-vanishing-condition-subcritical} or~\eqref{E:sobolev-to-campanato-vanishing-condition-critical} yields \eqref{E:vanishing-nonfull}.
\\
   $(ii)\Rightarrow(i)$. This implication is trivial.
\end{proof}

\section{Proofs of embeddings for special classes of Sobolev type spaces}\label{special}

We conclude by deriving the embeddings stated in   Example~\ref{BMOex} and Section~\ref{S:ex}  from our general results. The  embedding from  Example~\ref{BMOex} is the subject of the following proposition.

\begin{prop}{\rm\bf{ [Optimal Sobolev domain for BMO]}}\label{P:optimal-domain-for-BMO} Assume that $\Omega$
 be
a bounded John domain
 in $\mathbb R^n$.
    Let $m \in \mathbb N$ be such that $m \leq n$. Then, 
 \begin{equation}\label{E:orlicz-campanato-optimal-range2}
        W^{m}L^{\frac nm,\infty}(\Omega)\to\BMO(\Omega).
    \end{equation}
Moreover, $L^{\frac{n}{m},\infty}(\Omega)$ is the optimal (largest possible) rearrangement-invariant domain space in~\eqref{E:orlicz-campanato-optimal-range2}.
\end{prop}
\begin{proof}[Proof of Proposition~\ref{P:optimal-domain-for-BMO}]
    Set $k=0$ and $\varphi(r)=1$ for $r\in(0,\infty)$. Then $\mathcal L^{k, \varphi(\cdot)}(\Omega)=\BMO(\Omega)$. Note that~\eqref{E:Campanato-condition-on-phi} is satisfied. Thus, by Theorem~\ref{T:campanato-optimal-domain}, the optimal rearrangement-invariant domain space $\widehat X(\Omega)$ in
    \begin{equation*}
        W^m\widehat X(\Omega)\to\BMO(\Omega)
    \end{equation*}
    is associated with the rearrangement-invariant function norm given by
    \begin{equation*}
        \|f\|_{\widehat X(0,1)}=\sup_{r\in(0,1)}r^{\frac{1}{n}}\int_{r}^{1}s^{-1+\frac{m-1}{n}}f^{**}(s)\,\d s
    \end{equation*} 
    for $f \in L^0_+(0,1)$.
    Thus, it suffices   to
   show that
    \begin{equation}\label{E:domain-for-bmo-equivalence}
        \sup_{r\in(0,1)}r^{\frac{1}{n}}\int_{r}^{1}s^{-1+\frac{m-1}{n}}f^{**}(s)\,\d s\   \approx \ \sup_{r\in(0,1)}r^{\frac{m}{n}}f^{**}(r)
    \end{equation}
    with equivalence constants    independent on $f$.
\\
    Given $f \in L^0_+(0,1)$,  
       one has that
    \begin{equation}\label{feb202}
        \sup_{r\in(0,1)}r^{\frac{1}{n}}\int_{r}^{1}s^{-1+\frac{m-1}{n}}f^{**}(s)\,\d s
        \le \sup_{r\in(0,1)}r^{\frac{m}{n}}f^{**}(r) \sup_{r\in(0,1)}r^{\frac{1}{n}}\int_{r}^{1}s^{-1-\frac{1}{n}}\,\d s \   \approx \ \sup_{r\in(0,1)}r^{\frac{m}{n}}f^{**}(r)
    \end{equation}
    with equivalence constants    depending only on $n$. Conversely,
    \begin{equation*}\label{feb201}
        r^{\frac{1}{n}}\int_{r}^{1}s^{-1+\frac{m-1}{n}}f^{**}(s)\,\d s
        \ge r^{\frac{1}{n}}\int_{r}^{2r}s^{-1+\frac{m-1}{n}}f^{**}(s)\,\d s
        \ge r^{\frac{1}{n}} f^{**}( 2r)\int_{r}^{2r}s^{-1+\frac{m-1}{n}}\,\d s
       \   \approx \ r^{\frac{m}{n}}f^{**}(2r)
    \end{equation*}
    for $r\in \left(0,\frac12\right)$,
    with equivalence constants    depending only on $n$. Since
    \begin{equation*}
\sup_{r\in\left(0,\frac12\right)}r^{\frac{m}{n}}f^{**}(2r)\   \approx \ \sup_{r\in(0,1)}r^{\frac{m}{n}}f^{**}(r),
    \end{equation*}
    with equivalence constants   depending only on $m$ and $n$, the reverse inequality in~\eqref{feb202} follows. 
    This establishes~\eqref{E:domain-for-bmo-equivalence} and completes the proof.
\end{proof}

\begin{proof}[Proof of Theorem~\ref{T:orlicz-morrey-optimal-range}] The result follows via Theorem \ref{T:morrey-optimal-range}, since~\cite[Lemma~2]{Cia:96} tells us that
$$  
    \|s^{-1+\frac{m}{n}}\chi_{(r^n,1)}(s) \|_{L^{\widetilde A}(0,1)} \  \approx \  \frac{1}{ r^{ n-m }E_{m}^{-1}\left(r^{-n}\right)} \quad \ \text{near} \ 0.
    $$
\end{proof}

\begin{proof}[Proof of Theorem~\ref{P:orlicz-campanato-optimal-range}]
    The assertion follows   from Theorem~\ref{T:campanato-optimal-range}. Specifically, if    $k\in\{0,\dots,m-2\}$, one has to  make use of equation~\eqref{B.7}
    and~\cite[Lemma~2]{Cia:96} to deduce that
    \begin{align}\label{E:Orlicz-optCrange1}
        &\|s^{-1+\frac{m-k-1}{n}}\chi_{(r^n,1)}(s) \|_{L^{\widetilde A}(0,1)} \  \approx \  \frac{1}{ r^{ m-k-n }E_{m,k}^{-1}\left(r^{-n}\right)}\quad \ \text{near} \ 0.
        \end{align}
    If    $k=m-1$, one can exploit the fact that, by equation \eqref{young with conj},
        \begin{align}
        \label{E:Orlicz-optCrange2}
        \|\chi_{(0,r^n)}\|_{L^{\widetilde A}(0,1)} \   \approx \ \frac{1}{\widetilde A^{-1}\left(r^{-n}\right)} \ \approx \ r^{n}A^{-1}(r^{-n}) \quad \ \text{near} \ 0.
    \end{align}
\end{proof}

 \par\noindent {\bf Data availability statement.} Data sharing not applicable to this article as no datasets were generated or analysed during the current study.

\section*{Compliance with Ethical Standards}\label{conflicts}

\smallskip
\par\noindent
{\bf Funding}. This research was partly funded by:
\\ (i)\ Research Project   of the University of Salerno
  \lq \lq Regularity for PDEs and  functional methods'',
grant number 300396FRB22SOFTO  (P.~Cavaliere);
\\ (ii)\  GNAMPA   of the Italian INdAM - National Institute of High Mathematics (grant number not available)  (P.~Cavaliere, A.~Cianchi);
\\ (iii)\ Research Project   of the Italian Ministry of Education, University and
Research (MIUR) Prin 2017 \lq \lq Direct and inverse problems for partial differential equations: theoretical aspects and applications'',
grant number~201758MTR2 (A.~Cianchi);
\\ (iv)\ Research Project   of the Italian Ministry of Education, University and
Research (MIUR) Prin~2022 \lq \lq Partial differential equations and related geometric-functional inequalities'',
grant number 20229M52AS, cofunded by~PNRR (A.~Cianchi);
\\  (v)\  Grant number~23-04720S of the Czech Science
Foundation  (L.~Pick, L.~Slav\'ikov\'a);
\\ (vi)\  Primus research programme PRIMUS/21/SCI/002 of Charles University (L.~Slav\'ikov\'a);
\\ (vii)\  Charles University Research Centre program number~UNCE/24/SCI/005 (L.~Slav\'ikov\'a).

\bigskip
\par\noindent
{\bf Conflict of Interest}. The authors declare that they have no conflict of interest.

\end{document}